\documentclass{amsart}
\usepackage[margin=1in]{geometry}
\usepackage{graphicx} 
\usepackage{enumerate}
\usepackage{appendix}
\usepackage{amsthm,xcolor,soul}
\usepackage[colorlinks=true,citecolor=black,linkcolor=black,urlcolor=blue]{hyperref}
\newcommand{\seqnum}[1]{\href{http://oeis.org/#1}{\underline{#1}}}
\usepackage{amssymb}

\usepackage{mathtools,cleveref}
\newtheorem{theorem}{Theorem}[section]
\newtheorem{conjecture}[theorem]{Conjecture}
\newtheorem{proposition}[theorem]{Proposition}
\newtheorem{corollary}[theorem]{Corollary}
\newtheorem{lemma}[theorem]{Lemma}
\newtheorem{innercustomthm}{Theorem}
\newenvironment{customthm}[1] {\renewcommand\theinnercustomthm{#1}\innercustomthm}
  {\endinnercustomthm}

\newtheorem{innercustomprop}{Proposition}
\newenvironment{customprop}[1]{\renewcommand\theinnercustomprop{#1}\innercustomprop}
  {\endinnercustomprop}
  
\theoremstyle{definition}
\newtheorem{definition}[theorem]{Definition}
\newtheorem{example}[theorem]{Example}
\newtheorem{remark}[theorem]{Remark}

\usepackage{tikz}
\usetikzlibrary{calc}
\DeclareMathOperator{\BAS}{BAS}

\usepackage{thm-restate}

\definecolor{SJUBlue}{RGB}{0, 160, 230}
\definecolor{bluegreen}{RGB}{144,144,255}
\definecolor{niceblue}{HTML}{39A0FF}
\definecolor{legiblegreen}{RGB}{16, 194, 105}

\newcommand{\hroot}{\tilde{\alpha}}
\newcommand{\calA}{\mathcal{A}}
\newcommand{\calP}{\mathcal{P}}
\newcommand{\calH}{\mathcal{H}}
\newcommand*{\C}{\mathbb{C}}
\newcommand*{\Z}{\mathbb{Z}}
\newcommand*{\R}{\mathbb{R}}
\newcommand{\defn}[1]{\textbf{#1}}

\usepackage{cleveref}

\definecolor{c1}{HTML}{0085FF}
\definecolor{c2}{HTML}{FF9000}
\definecolor{c3}{HTML}{90D300}
\definecolor{c4}{HTML}{9728A1}

\newcommand{\precdot}{\prec\mathrel{\mkern-5mu}\mathrel{\cdot}}

\usepackage[colorinlistoftodos]{todonotes}

\usepackage{amsmath}

\author[Anderson]{Portia X. Anderson}
\address[P.~X. Anderson]{Department of Mathematics, Cornell University, Ithaca, NY 14853}
\email{\textcolor{blue}{\href{mailto:pxa2@cornell.edu}{pxa2@cornell.edu}}}

\author[Banaian]{Esther Banaian}
\address[E.~Banaian]{Department of Mathematics, University of California, Riverside, Riverside, California, 92521}
\email{\textcolor{blue}{\href{mailto:estherbanaian@gmail.com}{estherbanaian@gmail.com}}}

\author[Ferreri]{Melanie J. Ferreri}
\address[M. ~J. Ferreri]{Department of Mathematics, College of William \& Mary, Williamsburg, VA, 23187}
\email{\textcolor{blue}{\href{mailto:mjferreri@wm.edu}{mjferreri@wm.edu}}}

\author[Goff]{Owen C. Goff}
\address[O.~C.~Goff]{Department of Mathematics, University of Wisconsin-Madison, Madison, WI 53706}
\email{\textcolor{blue}{\href{mailto:ogoff@wisc.edu}{ogoff@wisc.edu}}}

\author[Hadaway]{Kimberly P. Hadaway}
\address[K.~P.~Hadaway]{Department of Mathematics, Iowa State University, Ames, IA, 50010}
\email{\textcolor{blue}{\href{mailto:kph3@iastate.edu}{kph3@iastate.edu}}}

\author[Harris]{Pamela E. Harris}
\author[Harry]{Kimberly J. Harry}
\address[P.~E.~Harris, K.~J.~Harry]{Department of Mathematical Sciences, University of Wisconsin-Milwaukee, Milwaukee, WI 53211}
\email{\textcolor{blue}{\href{mailto:peharris@uwm.edu}{peharris@uwm.edu}} \textcolor{blue}{\href{mailto:kjharry@uwm.edu}{kjharry@uwm.edu}}}

\author[Mayers]{Nicholas Mayers}
\address[N.~Mayers]{Department of Mathematics, North Carolina State University, Raleigh, NC, 27605}
\email{\textcolor{blue}{\href{mailto:nmayers@ncsu.edu}{nmayers@ncsu.edu}}}

\author[Wang]{Shiyun Wang}
\address[S.~Wang]{Department of Mathematics, University of Minnesota Twin Cities, MN, 55455}
\email{\textcolor{blue}{\href{mailto:wang8406@umn.edu}{wang8406@umn.edu}}}

\author[Wilson]{Alexander N.~Wilson}
\address[A.~N.~Wilson]{Department of Mathematics, Oberlin College, Oberlin, OH 44074}
\email{\textcolor{blue}{\href{mailto:awilson6@oberlin.edu}{awilson6@oberlin.edu}}}

\title[Weyl alternation sets are order ideals]{The support of Kostant's weight multiplicity formula is an order ideal in the weak Bruhat order
}
\begin{document}

\begin{abstract}
    For integral weights $\lambda$ and $\mu$ of a classical simple Lie algebra $\mathfrak{g}$, Kostant's weight multiplicity formula gives the multiplicity of the weight $\mu$ in the irreducible representation with highest weight $\lambda$, which we denote by $m(\lambda,\mu)$. 
    Kostant's weight multiplicity formula is an alternating sum over the Weyl group of the Lie algebra whose terms are determined via a vector partition function.
    The Weyl alternation set $\calA(\lambda,\mu)$ is the set of elements of the Weyl group that contribute nontrivially to the multiplicity $m(\lambda,\mu)$. 
    In this article, we prove that Weyl alternation sets are order ideals in the weak Bruhat order of the corresponding Weyl group. 
    Specializing to the Lie algebra $\mathfrak{sl}_{r+1}(\mathbb{C})$, we give a complete characterization of the Weyl alternation sets $\calA(\hroot,\mu)$, where $\hroot$ is the highest root and $\mu$ is a 
    negative root, answering a question of Harry posed in 2024. 
    We also provide some enumerative results that pave the way for our future work, where we aim to prove Harry's conjecture that the $q$-analog of Kostant's weight multiplicity formula is $m_q(\hroot,\mu)=q^{r+j-i+1}+q^{r+j-i}-q^{j-i+1}$ when $\mu=-(\alpha_i+\alpha_{i+1}+\cdots+\alpha_{j})$ is a negative root of $\mathfrak{sl}_{r+1}(\mathbb{C})$.
\end{abstract}

\maketitle
\section{Introduction}

 In this article, we are concerned with the computation of weight multiplicities using Kostant's weight multiplicity formula. 
 For a classical simple Lie algebra $\mathfrak{g}$ over $\mathbb{C}$ of rank $r\geq 1$, let $\Delta=\{\alpha_1,\alpha_2,\ldots,\alpha_r\}$ denote its set of simple roots,
    $\Phi$ denote its set of roots,
 $\Phi^+$ denote its set of positive roots, $\Phi^-$ denote its set of negative roots, $\tilde\alpha$
denote its highest root, and let $\rho=\frac{1}{2}\sum_{\alpha\in\Phi^+}\alpha$.
    Moreover, let $W$ denote the Weyl group associated to $\mathfrak{g}$,
    which is generated by reflections denoted as $s_1, s_2,\ldots, s_r$, through hyperplanes orthogonal to the simple roots $\alpha_1,\alpha_2,\ldots,\alpha_r$, respectively. 
If $\sigma\in W$, then $\ell(\sigma)$ denotes the length of $\sigma$.
Now, letting $L(\lambda)$ denote the irreducible highest weight representation of $\mathfrak{g}$ with highest weight $\lambda$; 
Kostant's weight multiplicity formula gives a way to compute the multiplicity of a weight $\mu$ in $L(\lambda)$, denoted $m(\lambda, \mu)$. Using an alternating sum over the Weyl group that involves the Kostant partition function, as given in \cite[Theorem on p.\ 589]{Kostant}, one has
\begin{align}
    m(\lambda, \mu)=\sum_{\sigma \in W} (-1)^{\ell(\sigma)}\wp(\sigma(\lambda+\rho)-\mu - \rho)\label{KWMF}
\end{align}
where 
$\wp$ denotes Kostant's partition function. That is, $\wp$ is the function from the set of weights to $\Z_{\geq0}=\{0,1,2,3,\ldots\}$
for which $\wp(\xi)$ is the number of ways the weight $\xi$ can be expressed as a linear combination of the positive roots of $\mathfrak{g}$ with nonnegative integer coefficients.

Two major challenges arise in computing weight multiplicities via Kostant's formula: 
\begin{enumerate}
    \item[1)] The number of terms grows factorially in the rank of the Lie algebra {as the order of the Weyl group of a classical Lie algebra is either $r!$ or exponential times $r!$} and
    \item[2)] There are no general closed formulas for the value of the Kostant partition function.
\end{enumerate} 
However, in practice, many of the terms are zero and contribute trivially to the multiplicity \cite{Cochet}. 
Hence, this motivates work to characterize and enumerate the Weyl group elements that contribute nontrivially to the multiplicity. This set is referred to as the Weyl alternation set and is defined formally below.

\begin{definition}[Definition 1.1 in \cite{PHThesisPublication}]
    For $\lambda$ and $\mu$ integral weights of $\mathfrak{g}$, the \textbf{Weyl alternation set} is
    \begin{equation}
        \mathcal{A}(\lambda, \mu) = \{\sigma \in W ~|~ 
 \wp(\sigma(\lambda+\rho)-\mu - \rho)>0\}.
    \end{equation}
\end{definition} 
Note that $\sigma\in W$ is in $\mathcal{A}(\lambda,\mu)$ if and only if $\sigma(\lambda+\rho)-\mu - \rho$ can be written as a sum of positive roots.
The Weyl alternation set allows one to reduce the computation of $m(\lambda,\mu)$ to be done over $\mathcal{A}(\lambda,\mu)$ rather than over $W$, which greatly reduces the computation involved\footnote{For example, \Cref{thm:gen_func} and the theory of rational generating functions (see \cite[Chapter 4]{EC1}) show that the size of the alternation set for the multiplicity of a root space in the adjoint representation in type $A_r$ grows as a polynomial times an exponential in the rank $r$, whereas the Weyl group grows as a factorial.}. 
We remark that even in cases where the multiplicity is known,
such as when the highest root $\lambda$ and a root $\mu$ produce the multiplicity of one, 
determining how many terms contribute to this multiplicity and what each of the contributing terms are
is very often unknown.
Work determining Weyl alternation sets for a variety of pairs of weights appears in the literature and includes characterizing Weyl alternation sets in both classical and exceptional Lie algebras
\cite{MR4015857,MR4747948,MR4452794,MR3792167,MR4707357,MR3774802,MR3825932,Lauren,MR3730711,MR4747948,MR4025421}.
Of note is the work of 
Harris, which establishes that if $\hroot$ is the highest root of $\mathfrak{sl}_{r+1}(\mathbb{C})$, then the Weyl alternation set $\mathcal{A}(\hroot,0)$ consists of products of certain commuting simple reflections and hence $|\mathcal{A}(\hroot,0)|=F_{r}$, the $r$th Fibonacci number \cite[Theorem~1.2]{PHThesisPublication}.
Harry establishes that if $\hroot$ is the highest root and $\mu=\alpha_i+\alpha_{i+1}+\cdots+\alpha_j$ is a height 
$j-i+1$ positive root of $\mathfrak{sl}_{r+1}(\mathbb{C})$, then the Weyl alternation set $\mathcal{A}(\hroot,\mu)$ again consists of certain products of commuting simple reflections and, hence, $|\mathcal{A}(\hroot,\mu)|=F_i \cdot F_{r-j+1}$ \cite[Corollary~2.5]{Harry}.

Motivated by the work of Harris and Harry, in this article, we prove general structural results about the Weyl alternation set. For example, we establish the following theorem elucidating the poset structure of $\calA(\lambda,\mu)$.

\begin{customthm}{\ref{prop:ideal_conditions_on_lambda}}
    Let $\lambda$ be an integral dominant weight of a simple Lie algebra $\mathfrak{g}$ with Weyl group $W$. Then for any weight $\mu$, the Weyl alternation set $\calA(\lambda,\mu)$ is a (possibly empty) order ideal in the left and right weak Bruhat orders of $W$.
\end{customthm}

Next, we generalize the results of Harris and Harry by replacing simple transpositions with elements having a property called \textit{connected influence} and replacing the commuting condition with a condition called \textit{pairwise independence} 
(see Definitions~\ref{def:ind} and~\ref{def:independence}), 
leading to the following result. 

\begin{customthm}{\ref{thm:characterize_BAS}}
    Let $\lambda$ be a dominant integral weight of a simple Lie algebra $\mathfrak{g}$ and $\mu$ a weight such that $\calA(\lambda,\mu)$ is nonempty. Then
    
    \begin{enumerate}[(1)]
        \item there exists a unique subset $S\subseteq \calA(\lambda,\mu)$ where $1\notin S$ such that each $b\in S$ has connected influence and any element $\sigma\in \calA(\lambda,\mu)$ can be written as a product of elements from a pairwise independent subset of $S$, and 
        \item there is a bijection between elements of $\calA(\lambda,\mu)$ and pairwise independent subsets of $S$. In particular, each pairwise independent subset of $S$ corresponds to the element in $\calA(\lambda,\mu)$ that can be written as the product of elements of the subset.
    \end{enumerate}
\end{customthm}

\setcounter{section}{1}
\setcounter{theorem}{0}
In \Cref{thm:characterize_BAS}(1), the products could be over possibly empty subsets of $S$. 
Using these results, we specialize to the Lie algebra of type $A_r$ where the positive roots are given by $\alpha_{i,j}=\alpha_i+\alpha_{i+1}+\cdots+\alpha_j$ for $1\leq i\leq j\leq r$. 
Specifically, in \Cref{sec:NegAlpha}, we characterize the Weyl alternation sets $\mathcal{A}_r(\hroot,\mu)$, where we use the subscript $r$ to denote the rank, $\hroot$ to denote the highest root, and $\mu$ to denote a negative root of $\mathfrak{sl}_{r+1}(\mathbb{C})$. 
In addition, in \Cref{sec:enumerative}, we prove that the cardinalities of these Weyl alternation sets satisfy the two-term recurrence relation defining the Fibonacci numbers detailed in the following result.

\begin{customprop} {\ref{prop:RecurrenceOnSizesFixedRoot}}
Let $r\geq 3$ and 
 $1 \leq i \leq j \leq r-2 $. If $\mu = -\alpha_{i,j}$, 
then
\[
\vert \mathcal{A}_r(\hroot, \mu) \vert = \vert \mathcal{A}_{r-1}(\hroot, \mu) \vert  + \vert \mathcal{A}_{r-2}(\hroot, \mu) \vert. 
\]
\end{customprop}

We remark that the initial values of the recurrence relation depend on $\mu$, and we present some examples in \Cref{sec:enumerative}.

Using the recurrence in \Cref{prop:RecurrenceOnSizesFixedRoot}, we provide the generating function for the sizes of Weyl alternation sets for all type $A$ negative roots.

\begin{customthm}{\ref{thm:gen_func}}
    If $\hroot$ is the highest root and $\mu=-\alpha_{i,j}$ with $1\leq i\leq j\leq r$
 is a negative root of the Lie algebra of type $A_r$, then we have the generating function
\begin{align*}
\sum_{1\leq i\leq j\leq r} \vert \mathcal{A}_r(\hroot, -\alpha_{i,j})\vert x^r s^it^j &= \frac{1}{t(1-x-x^2)}\left((1-x)t\calH(xt,s)+\calP(xt,s)-\frac{xst}{1-xst-(xst)^2}\right)\\\intertext{where}
\calH(x,s)&=\frac{xs(x^5s+3x^4s-xs+x^2+2x+1)}{(1-x-x^2-3x^3-x^4)(1-xs-(xs)^2)}\\\intertext{and}
\calP(x,s)&=\frac{xs(x^4s+3x^3s+x+1)}{(1-x-x^2-3x^3-x^4)(1-xs-(xs)^2)}.
\end{align*} 
\end{customthm}

The remainder of the article is organized as follows. In \Cref{sec:background}, we provide the necessary background to make our approach precise. In \Cref{sec:ideals}, we prove \Cref{prop:ideal_conditions_on_lambda} establishing that the Weyl alternation sets are order ideals in the weak Bruhat order. 
Then, we specialize our study to the Lie algebra of type $A$, and in \Cref{sec:NegAlpha}, we characterize all of the Weyl alternation sets $\calA(\hroot,\mu)$ where $\hroot$ is the highest root and $\mu$ is a negative root. 
In \Cref{sec:enumerative}, we give recursive formulas and generating functions for the cardinalities of the Weyl alternation sets considered in the previous section. 
We conclude with \Cref{sec:future} in which we detail our future work where we aim to establish a conjecture of Harry for the $q$-multiplicity of a negative root in the adjoint representation of $\mathfrak{sl}_{r+1}(\mathbb{C})$.

\section{Background}\label{sec:background}

In this section, we cover the requisite preliminaries from the theories of Lie algebras, root systems, and posets.

 Let $\mathfrak{g}$ be a semi-simple Lie algebra and let $\mathfrak{h}$ be a Cartan subalgebra of $\mathfrak{g}$. A \defn{weight} of a representation $\sigma:\mathfrak{g}\rightarrow \mathfrak{gl}(V)$ is a linear functional $\lambda:\mathfrak{h}\rightarrow \mathbb{C}$ with a corresponding \defn{weight space} $V_\lambda$, where
$$V_\lambda:=\{v\in V\mid \forall H\in\mathfrak{h},\, \sigma(H)(v)=\lambda(H)v\}.$$ The nonzero elements of $V_\lambda$ are called \defn{weight vectors}. 
We can choose an inner product $(\cdot,\cdot)$ on $\mathfrak{h}^*$ that is invariant under the action of the Weyl group (defined below), and this allows us to identify weights with elements of a Euclidean space $\mathbb{R}^d$ with dimension $d$. 
A \textbf{root} is a weight of the adjoint representation, and the collection of roots forms a \textbf{root system} $\Phi\subseteq \mathbb{R}^d$. 
There is a subset $\Delta=\{\alpha_1,\alpha_2,\ldots,\alpha_r\} \subseteq \Phi$ of \textbf{simple roots} (where $r=\dim(\mathfrak{h})$ is the rank of $\mathfrak{g}$), which acts as an integral basis for $\Phi$.
The \defn{positive roots} $\Phi^+\subseteq \Phi$ are those that can be written as a linear combination of the simple roots with nonnegative integer coefficients; the \defn{height} of a positive root is the sum of these coefficients. Similarly, the \defn{negative roots} $\Phi^- \subseteq \Phi$ are those which can be written as a nonpositive, integral linear combination of simple roots and their height can be defined as the absolute value of the sum of the coefficients.

For weights $\lambda$ and $\mu$, we write $\mu\geq\lambda$ and say that $\mu$ is \defn{higher} than $\lambda$ if $\mu-\lambda$ can be expressed as a nonnegative linear combination of positive roots. We call a weight $\lambda$ \defn{integral} if $2\frac{(\lambda,\alpha)}{(\alpha,\alpha)}\in\mathbb{Z}$ for all roots $\alpha$, while we call $\lambda$ \defn{dominant} if $(\lambda,\alpha)\geq0$ for each positive root $\alpha$. Every dominant integral element is the unique highest weight of a finite-dimensional irreducible representation of $\mathfrak{g}$, and via this correspondence, the set of dominant integral elements is in bijection with the set of isomorphism classes of all such representations.

The data of the root system can be summarized in a graph called a \defn{Coxeter--Dynkin diagram} (sometimes referred to as a \textbf{Coxeter diagram}). 
This graph has vertices labeled by $1,2,\ldots, r$ and an edge between $i$ and $j$ if and only if $\alpha_i$ and $\alpha_j$ are not orthogonal. The edge between $i$ and $j$ is usually labeled by a positive integer $m_{i,j}$ for which the angle between $\alpha_i$ and $\alpha_j$ is given by $\frac{\pi}{m_{i,j}}$, but we suppress such labels in this paper.

The \textbf{Weyl group} $W$ associated to $\mathfrak{g}$ is generated by simple reflections $s_1, s_2,\ldots,s_r$ where $s_i$ reflects vectors through the hyperplane orthogonal to $\alpha_i$. 
That is, \begin{align*}
    s_i(v)=v-2\frac{(\alpha_i,v)}{(\alpha_i,\alpha_i)}\alpha_i,
\end{align*} where $(\cdot,\cdot)$ is the standard inner product on $\R^d$. 
The inner product $(\cdot,\cdot)$ is invariant under the action of the Weyl group.
Two reflections $s_i$ and $s_j$ commute if and only if $\alpha_i$ and $\alpha_j$ are orthogonal---that is, if $i$ and $j$ are nonadjacent in the Coxeter--Dynkin diagram. 
When viewing the Weyl group as an abstract Coxeter group, this action on the roots coincides with what is called the \defn{geometric representation} of the group.

To the Weyl group $W$, one can associate the ``right and left weak (Bruhat) orders".
Recall that a \defn{poset} $(\mathcal{P},\preceq)$ consists of a set $\mathcal{P}$ along with a binary relation $\preceq$ between the elements of $\mathcal{P}$ that is 
reflexive, anti-symmetric, and transitive. 
For $x,y\in\mathcal{P}$, if $x\preceq y$ and $x\neq y$, then we write $x\prec y$. 
In the case $x\prec y$ and there exists no $z\in \mathcal{P}$ satisfying $x\prec z\prec y$, then $x\prec y$ is a \defn{covering relation}, denoted $x\precdot y$, and we say that $y$ covers $x$. Covering relations are used to define a visual representation of $\mathcal{P}$ called the \defn{Hasse diagram}---a graph whose vertices correspond to elements of $\mathcal{P}$ and whose edges correspond to covering relations. For further details concerning posets we recommend \cite{EC1}. 

Writing elements of a Weyl group $\sigma\in W$ as \defn{reduced words}, i.e., minimal length expressions for $\sigma$ as a product of generators $s_i$, and denoting the length of such expressions by $\ell(\sigma)$, the posets on $W$ of interest here are defined as follows.

\begin{definition}
The \defn{right weak (Bruhat) order} $(W,\leq_R)$ on the Weyl group $W$ is the transitive closure of the covering relations \[\sigma\lessdot_R \sigma s_i \text{ whenever } \ell(\sigma s_i)>\ell(\sigma).
\] 
The \defn{left weak (Bruhat) order} $(W,\leq_L)$ is similarly defined by covering relations \[\sigma\lessdot_L s_i\sigma \text{ whenever } \ell(s_i\sigma)>\ell(\sigma).\] 
\end{definition}

The right (resp. left) weak order can be equivalently given as $\sigma<_R\tau$ (resp. $\sigma<_L\tau$) if in some reduced expression of $\tau$ a reduced expression of $\sigma$ appears as a prefix (resp. suffix). 

\begin{example} \Cref{fig:HasseDiagram} illustrates the Hasse diagram of the right weak order on the type $A_3$ Weyl group.
    \begin{figure}[htp]
    \centering
        \begin{tikzpicture}
    [scale=1.2,auto=left]

        \node (4321) at (0,3) {$s_1s_2s_3s_1s_2s_1$};

    \node (4312) at (-1.75,2) {$s_2s_3s_2s_1s_2$};
    \node (4231) at (0,2) {$s_1s_2s_3s_2s_1$};
    \node (3421) at (1.75,2) {$s_1s_2s_3s_1s_2$};

    \node (4132) at (-3,1) {$s_2s_3s_2s_1$};
    \node (4213) at (-1.5,1) {$s_3s_1s_2s_1$};
    \node (3412) at (0,1) {$s_2s_3s_1s_2$};
    \node (2431) at (1.5,1) {$s_1s_2s_3s_2$};
    \node (3241) at (3,1) {$s_1s_2s_1s_3$};

    \node (1432) at (-3.75,0) {$s_2s_3s_2$};
    \node (4123) at (-2.25,0) {$s_3s_2s_1$};
    \node (2413) at (-.75,0) {$s_1s_3s_2$};
    \node (3142) at (.75,0) {$s_2s_3s_1$};
    \node (3214) at (2.25,0) {$s_1s_2s_1$};
    \node (2341) at (3.75,0) {$s_1s_2s_3$};

    \node (1423) at (-3,-1) {$s_3s_2$};
    \node (1342) at (-1.5,-1) {$s_2s_3$};
    \node[text=c1] (2143) at (0,-1) {\fbox{$s_1s_3$}};
    \node (3124) at (1.5,-1) {$s_2s_1$};
    \node (2314) at (3,-1) {$s_1s_2$};
    
    \node[text=c1] (1243) at (-1.5,-2) {\fbox{$s_3$}};
    \node[text=c1] (1324) at (0,-2) {\fbox{$s_2$}};
    \node[text=c1] (2134) at (1.5,-2) {\fbox{$s_1$}};
    
    \node[text=c1] (1234) at (0,-3) {\fbox{$1$}};

    \foreach \from/\to in {
        1243/1423,
        1324/1342,
        1324/3124,
        2134/2314,
        1423/1432,
        1423/4123,
        1342/1432,
        1342/3142,
        2143/2413,
        3124/3142,
        3124/3214,
        2314/3214,
        2314/2341,
        1432/4132,
        4123/4132,
        4123/4213,
        2413/4213,
        2413/2431,
        3142/3412,
        3214/3241,
        2341/3241,
        2341/2431,
        4132/4312,
        4213/4231,
        3412/4312,
        3412/3421,
        2431/4231,
        3241/3421,
        4312/4321,
        4231/4321,
        3421/4321}{
        \draw (\from) -- (\to);}

    \foreach \from/\to in {
        1234/1243,
        1234/1324,
        1234/2134,
        2134/2143,
        1243/2143}{
    \draw[c1, ultra thick] (\from) -- (\to);}
\end{tikzpicture}
\caption{Hasse diagram of the right weak order on the type $A_3$ Weyl group with the elements of the order ideal $I=\{1,s_1,s_2,s_3,s_1s_3\}$ boxed and in blue text.}
\label{fig:HasseDiagram}
\end{figure}
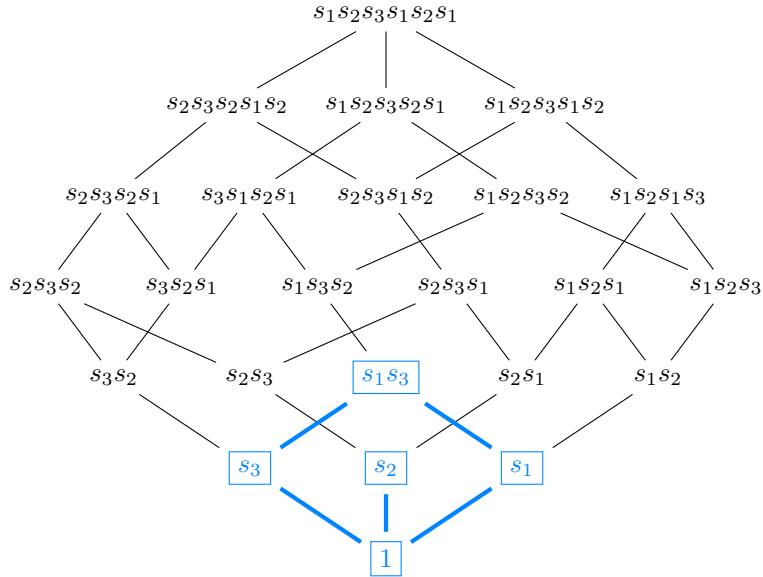
\end{example}

As noted in the introduction, certain subsets of the left and right Bruhat orders, called ``order ideals", play an important role ongoing in this study. Following \cite[p.\ 282]{EC1}, we define an \defn{order ideal} to be a subposet $I$ of a poset $\mathcal{P}$ with the property that if $y\in I$ and $x\leq y$, then $x\in I$.

\begin{example}\label{Ex: right weak order} 
The set $I=\{1,s_1,s_2,s_3,s_1s_3\}$ is 
an order ideal of the right weak order on the type $A_3$ Weyl group.
This order ideal is generated by the elements $s_2$ and $s_1s_3$ which are the maximal elements in $I$.
In \Cref{fig:HasseDiagram}, we highlight the order ideal 
$I$ within the Hasse diagram of the right weak order on the type $A_3$ Weyl group. 
 \end{example}

\section{Weyl alternation sets are order ideals}\label{sec:ideals}
In this section, we work with a general simple Lie algebra $\mathfrak{g}$, and our main focus is the structure of the associated Weyl alternation sets.
Our first main result establishes that $\calA(\lambda,\mu)$ for an integral dominant weight $\lambda$ is an order ideal in both the left and right weak Bruhat orders. 
\begin{theorem}\label{prop:ideal_conditions_on_lambda}
    Let $\lambda$ be an integral dominant weight of a simple Lie algebra $\mathfrak{g}$ with Weyl group $W$. Then for any weight $\mu$, the Weyl alternation set $\calA(\lambda,\mu)$ is a (possibly empty) order ideal in the left and in the right weak Bruhat orders of $W$.
\end{theorem}
\begin{proof}
Suppose that $\calA(\lambda,\mu)\neq\emptyset$; otherwise, $\calA(\lambda,\mu)$ is vacuously an order ideal. To show that $\calA(\lambda,\mu)$ is an order ideal with respect to a partial order $\leq$, it suffices to establish the following: if $\sigma\lessdot\tau$ (that is, $\tau$ covers $\sigma$ in the partial order), then the difference $\sigma(\lambda+\rho)-\tau(\lambda+\rho)$ is a linear combination of positive roots with nonnegative integer coefficients.
Indeed, this implies that $\sigma(\lambda+\rho)-\rho-\mu$ could be obtained from $\tau(\lambda+\rho)-\rho-\mu$ by adding some number of simple roots with positive coefficients. 
Consequently, if $\wp(\tau(\lambda+\rho)-\rho-\mu)>0$, then it would then be the case that $\wp(\sigma(\lambda+\rho)-\rho-\mu)>0$.
For example, a partition of the latter could be obtained by adding these same simple roots to the partition of the former, each as their own part.
Hence, $\tau\in \calA(\lambda,\mu)$ would imply that $\sigma\in \calA(\lambda,\mu)$, i.e., $\calA(\lambda,\mu)$ would be an order ideal.

First, we consider the right Bruhat order.
Suppose that $\sigma\lessdot_R\tau$. 
By definition, $\tau=\sigma s_i$ for some simple reflection $s_i\in W$ and $\ell(\sigma s_i)>\ell(\sigma)$. 
We now compute
\begin{align*}
      \sigma(\lambda+\rho)-\sigma s_i(\lambda+\rho)&=\sigma(\lambda+\rho)-\left(\sigma(\lambda+\rho)-2\frac{(\lambda+\rho,\alpha_i)}{(\alpha_i,\alpha_i)}\sigma(\alpha_i)\right)=2\frac{(\lambda+\rho,\alpha_i)}{(\alpha_i,\alpha_i)}\sigma(\alpha_i).
    \end{align*}

    By \cite[Equation 4.25]{bjorner2005combinatorics}, $\ell(\sigma s_i)>\ell(\sigma)$ implies that $\sigma(\alpha_i)\in\Phi^+$. Because $\lambda$ and $\rho$ are dominant integral weights, $2\frac{(\lambda+\rho,\alpha_i)}{(\alpha_i,\alpha_i)}$ is a nonnegative integer, so the difference $\sigma(\lambda+\rho)-\sigma s_i(\lambda+\rho)$ is indeed a linear combination of simple roots with nonnegative integer coefficients

Now, consider the left Bruhat order.
Suppose instead that $\sigma\lessdot_L\tau$. 
By definition $\tau=s_i\sigma$ and $\ell(s_i\sigma)>\ell(\sigma)$. 
We now compute
\begin{align*}
        \sigma(\lambda+\rho)-s_i \sigma(\lambda+\rho)&=\sigma(\lambda+\rho)-\left(\sigma(\lambda+\rho)-2\frac{(\sigma(\lambda+\rho), \alpha_i)}{(\alpha_i,\alpha_i)}\alpha_i\right)\\
        &=2\frac{(\sigma(\lambda+\rho), \alpha_i)}{(\alpha_i,\alpha_i)}\alpha_i\\
        &=2\frac{(\lambda+\rho,\sigma^{-1}(\alpha_i))}{(\sigma^{-1}(\alpha_i),\sigma^{-1}(\alpha_i))}\alpha_i,
    \end{align*}
    where the last equality holds because the inner product is invariant under the action of $W$ (for a proof, we point the reader to \cite[Equation 4.23]{bjorner2005combinatorics}).
    Moreover, $\ell(s_i\sigma)>\ell(\sigma)$ implies that $\ell(\sigma^{-1}s_i)>\ell(\sigma^{-1})$, and \cite[Equation 4.25]{bjorner2005combinatorics} tells us that $\sigma^{-1}(\alpha_i)$ is a positive root. Further, $\lambda$ and $\rho$ are dominant integral weights.
    Hence, $2\frac{(\lambda+\rho,\sigma^{-1}(\alpha_i))}{(\sigma^{-1}(\alpha_i),\sigma^{-1}(\alpha_i))}$ is a nonnegative integer, and the difference $\sigma(\lambda+\rho)-s_i\sigma(\lambda+\rho)$ is indeed a linear combination of simple roots with nonnegative integer coefficients.
\end{proof}

The following result is essential in our work to characterize Weyl group elements that are not contained in certain Weyl alternation sets.
Recall that a reduced word of $\sigma \in W$ is an expression of $\sigma$ as a minimal-length product of generators $s_i$.

\begin{corollary} \label{cor:subword_property}
     Let $\lambda$ be a dominant integral weight of a simple Lie algebra $\mathfrak{g}$. For any weight $\mu$, if a reduced word of $\sigma\in W$ is a consecutive subword of a reduced word of $\tau\in \calA(\lambda,\mu)$, then $\sigma\in\calA(\lambda,\mu)$.
\end{corollary}

\begin{proof}
Suppose $\sigma\in W$ and $\tau\in \calA(\lambda,\mu)$. If a reduced word of $\sigma$ is a consecutive subword of a reduced word of $\tau$, then $\tau=\tau_1\sigma\tau_2$ for some reduced words $\tau_1$ and $\tau_2$. Note, $\tau_1\sigma\leq_R \tau$. 
By \Cref{prop:ideal_conditions_on_lambda}, we have that $\tau_1\sigma\in\calA(\lambda,\mu)$. As $\sigma\leq_L \tau_1\sigma$, we similarly have that $\sigma\in\calA(\lambda,\mu)$.
\end{proof}

\begin{example}
    Using \Cref{fig:HasseDiagram}, let $\tau$ be the top element $s_1s_2s_3s_1s_2s_1$. We can also write $\tau$ as the reduced word $s_1s_2s_1s_3s_2s_1$. Since $\sigma=s_1s_3s_2$ is a consecutive subword of $\tau$ such that $\tau=s_1s_2\sigma s_1$, we have that $\tau\in\calA(\lambda,\mu)$ implies $\sigma\in \calA(\lambda,\mu)$ by \Cref{cor:subword_property}.
\end{example}

In the remainder of this section, we shift our attention to the elements comprising Weyl alternation sets. In particular, for $\lambda$ a dominant weight of a simple Lie algebra $\mathfrak{g}$ and $\mu$ a weight for which $\mathcal{A}(\lambda,\mu)$ is nonempty, we associate a unique subset $\BAS(\lambda,\mu)\subseteq \mathcal{A}(\lambda,\mu)$ consisting of what can be thought of as basic ``building blocks" for the elements of $\mathcal{A}(\lambda,\mu)$. The definition and study of these sets constitute the remainder of this section.

To define the collection $\BAS(\lambda,\mu)$, we introduce the following terms.

\begin{definition}\label{def:ind}
For an element $\sigma\in W$ of the Weyl group, define the \defn{influence} and \defn{extended influence} of $\sigma$, denoted $I(\sigma)$ and $\overline{I}(\sigma)$ respectively, by \begin{align*}
    I(\sigma)&=\{i:\text{$s_i$ is in a reduced word for $\sigma$}\},\mbox{ and}\\
    \overline{I}(\sigma)&=\{i:i\in I(\sigma)\text{ or $i$ is adjacent to some $j\in I(\sigma)$ in the Dynkin diagram}\}.
\end{align*}
We say that $I(\sigma)$ is \defn{connected} if the subgraph of the Dynkin diagram it induces is connected. 
\end{definition}

Note that, in type $A$, the extended influence $\overline{I}(\sigma)$ consists of all $i$ such that $i\in I(\sigma)$ or $i\pm1\in I(\sigma)$.

Next, using influence and extended influence, we define the following relation on elements of the Weyl group with respect to their influences and extended influences.

\begin{definition}  \label{def:independence} We say that $\sigma, \tau\in W$ are \defn{independent} if $I(\sigma)\cap\overline{I}(\tau)=\emptyset$. A set $S\subseteq W$ is called \defn{pairwise independent} if any pair of elements $\sigma\neq\tau$ in $S$ are independent.
\end{definition}

\begin{example} \label{ex:influence} Consider the type $A_7$ Dynkin diagrams in \Cref{fig:influence} with the influences \begin{align*}
    \textcolor{c4}{I(s_1s_2)=I(s_1s_2s_1)}&=\{1,2\},\\
    \textcolor{c2}{I(s_4s_6s_7)}&=\{4,6,7\},\\
    \textcolor{c1}{\overline{I}(s_1s_2)=\overline{I}(s_1s_2s_1)}&=\{1,2,3\},\mbox{ and}\\
    \textcolor{c3}{\overline{I}(s_4s_6s_7)}&=\{3,4,5,6,7\}
\end{align*} highlighted. 
Note, $I(s_1s_2)=I(s_1s_2s_1)$ is connected and $I(s_4s_6s_7)$ is disconnected. The elements $s_1s_2$ and $s_4s_6s_7$ are independent. 
    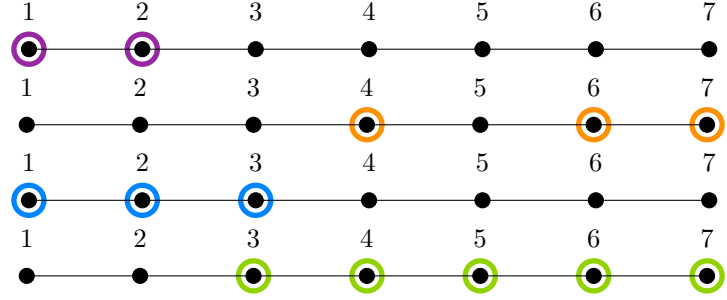
\begin{figure}
    \centering
\begin{tikzpicture}
    \foreach \i in {1,2,...,7} {
        \coordinate (v\i) at (\i*1.5,0);
    }

    \foreach \i in {1,2} {
        \node[circle, draw=c4, inner sep=4pt, line width=2pt] at (v\i) {};
    }

    \foreach \i in {1,2,...,7} {
            \node[circle, draw, fill=black, inner sep=2pt] at (v\i) {};
        \node at (\i*1.5,.5) {\i};
        }

    \draw (v1) -- (v7);
\end{tikzpicture}

\begin{tikzpicture}
    \foreach \i in {1,2,...,7} {
        \coordinate (v\i) at (\i*1.5,0);
    }


    \foreach \i in {4,6,7} {
        \node[circle, draw=c2, inner sep=4pt, line width=2pt] at (v\i) {};
    }

    \foreach \i in {1,2,...,7} {
            \node[circle, draw, fill=black, inner sep=2pt] at (v\i) {};
        \node at (\i*1.5,.5) {\i};
        }

    \draw (v1) -- (v7);
\end{tikzpicture}

\begin{tikzpicture}
    \foreach \i in {1,2,...,7} {
        \coordinate (v\i) at (\i*1.5,0);
    }

    \foreach \i in {1,2,3} {
        \node[circle, draw=c1, inner sep=4pt, line width=2pt] at (v\i) {};
    }

    \foreach \i in {1,2,...,7} {
            \node[circle, draw, fill=black, inner sep=2pt] at (v\i) {};
        \node at (\i*1.5,.5) {\i};
        }

    \draw (v1) -- (v7);
\end{tikzpicture}

\begin{tikzpicture}
    \foreach \i in {1,2,...,7} {
        \coordinate (v\i) at (\i*1.5,0);
    }

    \foreach \i in {3,4,5,6,7} {
        \node[circle, draw=c3, inner sep=4pt, line width=2pt] at (v\i) {};
    }

    \foreach \i in {1,2,...,7} {
            \node[circle, draw, fill=black, inner sep=2pt] at (v\i) {};
        \node at (\i*1.5,.5) {\i};
        }

    \draw (v1) -- (v7);
\end{tikzpicture}
\caption{The type $A_7$ Dynkin diagram with some influences and extended influences highlighted. }\label{fig:influence}
\end{figure}
\end{example}

\newcommand{\typeDfivemacro}[2]{\begin{tikzpicture}
    \foreach \i in {1,2,3} {
        \coordinate (v\i) at (\i*1.2,0);
    }

    \coordinate (v4) at (4.5, .7);
    \coordinate (v5) at (4.5, -.7);

    \foreach \i in {#2} {
        \node[circle, draw=#1, inner sep=4pt, line width=2pt] at (v\i) {};
    }

    \foreach \i in {1,2,3} {
            \node[circle, draw, fill=black, inner sep=2pt] at (v\i) {};
        \node at (\i*1.2,.5) {\i};
        }

         \node[circle, draw, fill=black, inner sep=2pt] at (v4) {};
        \node at (5,.7) {4};
         \node[circle, draw, fill=black, inner sep=2pt] at (v5) {};
        \node at (5,-.7) {5};

    \draw (v1) -- (v3);
    \draw (v3) -- (v4);
    \draw (v3) -- (v5);
     \path (current bounding box.south west) +(-.2,-.2) (current bounding box.north east) +(.2,.2);
\end{tikzpicture}}

\begin{example}
In \Cref{fig:anotherex of influence}, we consider the influence and extended influence of two elements in type $D_5$.
    \begin{table}[h!]
        \centering
        \begin{tabular}{c|c|c}
         $\sigma$ & $I(\sigma)$ & $\overline{I}(\sigma)$\\\hline
         \raisebox{1cm}{$s_3s_4$} & $\typeDfivemacro{c4}{3,4}$ & $\typeDfivemacro{c3}{2,3,4,5}$ \\ \hline
         \raisebox{1cm}{$s_2$}  & $\typeDfivemacro{c4}{2}$ & $\typeDfivemacro{c3}{1,2,3}$ 
         \end{tabular}
\caption{Examples of influence and extended influence for elements in the Weyl group of type $D_5$.}
\label{fig:anotherex of influence}
    \end{table}
    Note, $I(s_3s_4)\cap \overline{I}(s_2)=\{3\}\neq\emptyset$, so $s_3s_4$ and $s_2$ are not independent.
\end{example}

The following result concerning influence is useful in subsequent results.

\begin{lemma}\label{lem:properties_of_influence}
    Let $\sigma,\tau\in W$.
    \begin{enumerate}
        \item If $\sigma=s_{i_1}s_{i_2}\cdots s_{i_\ell}$ is a reduced word, then \begin{align*}
        I(\sigma)=\{i_j:1\leq j\leq\ell\}.
    \end{align*}
        \item If $I(\sigma)=I(\tau)$ and $\sigma(\alpha_i)=\tau(\alpha_i)$ for all $i\in I(\sigma)$, then $\sigma=\tau$.
        \item If $\{\sigma_1,\sigma_2,\ldots,\sigma_k\}\subseteq W$ is pairwise independent, then \begin{align*}
            I(\sigma_m)\subseteq I(\sigma_1\sigma_2\cdots\sigma_k)
        \end{align*} for all $1\leq m\leq k$.
    \end{enumerate}
\end{lemma}

\begin{proof} For item (1): This property follows from \cite[Corollary 2.2.3]{bjorner2005combinatorics}, which states that every reduced expression of $\sigma$ contains a reduced expression of $\tau$ if and only if some reduced expression of $\sigma$ contains a reduced expression of $\tau$.

    For item (2): Suppose $I(\sigma)=I(\tau)$ and let $I=I(\sigma)=I(\tau)$. Then $\sigma,\tau\in W_I$ where $W_I$ is the subgroup of $W$ generated by $\{s_i:i\in I\}$. Since the geometric representation of $W_I$ is isomorphic to the action of $W_I$ on $V_I=\mathbb{R}\{\alpha_i:i\in I\}$ (for a reference we recommend \cite[Section 5.5]{humphreys1990reflection}) and the geometric representation of any finite Coxeter group is faithful, $W_I$ acts faithfully on $V_I$. 
    Hence, if $\sigma(\alpha_i)=\tau(\alpha_i)$ for all $i\in I$, we have that $\sigma=\tau$.

    For item (3): Suppose that $\sigma_1,\sigma_2,\ldots,\sigma_k$ are reduced expressions and that the set $\{\sigma_1, \sigma_2,\ldots,\sigma_k\}$ is pairwise independent. 
    If $\sigma_1\sigma_2\cdots\sigma_k$ were not a reduced word, we could apply a sequence of nil-moves and braid-moves as in \cite[Theorem 3.3.1]{bjorner2005combinatorics} to arrive at a reduced expression. Since these operations only apply to strings of generators that are adjacent in the Dynkin diagram, each of these moves must occur within a $\sigma_i$. 
    But each $\sigma_i$ is assumed to be reduced, so $\sigma_1\sigma_2\cdots\sigma_k$ is a reduced expression. 
    Hence, if $s_i$ occurs in $\sigma_m$ for some $m, 1 \le m \le k$, then it occurs in the product $\sigma_1\sigma_2\cdots\sigma_k$ and property (3) follows.
\end{proof}

We now define the set $\BAS(\lambda,\mu)$.

\begin{definition}
    Let $\lambda$ be a dominant integral weight of a simple Lie algebra $\mathfrak{g}$ and $\mu$ a weight such that $\calA(\lambda,\mu)$ is nonempty. 
    Then $\BAS(\lambda,\mu)$ is the unique subset $S\subseteq \calA(\lambda,\mu)$ where $1\notin S$ such that each $b\in S$ has connected influence and any element $\sigma\in \calA(\lambda,\mu)$ can be written as a product of elements from a pairwise independent subset of $S$. 
    We refer to the elements of $\BAS(\lambda,\mu)$ as the \defn{basic allowable subwords} for the pair $\lambda$ and $\mu$.
\end{definition}

Next, we establish that the sets $\BAS(\lambda,\mu)$ are well-defined and that all elements of $\calA(\lambda,\mu)$ can be formed uniquely as products of elements belonging to pairwise independent subsets of $\BAS(\lambda,\mu)$.

\begin{theorem}\label{thm:characterize_BAS}
    Let $\lambda$ be a dominant integral weight of a simple Lie algebra $\mathfrak{g}$ and $\mu$ a weight such that $\calA(\lambda,\mu)$ is nonempty. Then
    
    \begin{enumerate}[(1)]
        \item there exists a unique subset $S\subseteq \calA(\lambda,\mu)$ where $1\notin S$ such that each $b\in S$ has connected influence and any element $\sigma\in \calA(\lambda,\mu)$ can be written as a product of elements from a pairwise independent subset of $S$, and 
        \item there is a bijection between elements of $\calA(\lambda,\mu)$ and pairwise independent subsets of $S$. In particular, each pairwise independent subset of $S$ corresponds to the element in $\calA(\lambda,\mu)$ that can be written as the product of elements of the subset.
    \end{enumerate}
\end{theorem}

\begin{proof}

    For item (1): Toward proving uniqueness, suppose that $S$ and $S'$ both satisfy the conditions in the statement of the theorem and let $b\in S$. 
    Since $b\in A(\lambda,\mu)$, we have that $b=b'_1b'_2\cdots b'_k$ for $\{b'_1,b'_2,\ldots,b'_k\}$ a pairwise independent subset of $S'$. By \Cref{lem:properties_of_influence} item (3), we have that $I(b'_m)\subseteq I(b)$ for all $1\leq m\leq k$. 
    Note that, since $I(b)$ must be connected as $b\in S$, if $I(b'_i)\neq I(b)$ for some $1\leq i\leq k$, then there must exist $b'_j$ for which $1\le j\neq i\le k$ and $I(b'_i)\cap \overline{I}(b'_j)\neq \emptyset$. 
    This contradicts our assumption that $\{b'_1,b'_2,\ldots,b'_k\}$ is a pairwise independent subset. 
    Thus, we must in fact have that $I(b)=I(b'_m)$ for all $1\leq m\leq k$. 
    Consequently, the only way that the set $\{b'_1,b'_2,\ldots,b'_k\}$ can be pairwise independent is if it consists of a single element $b'_1$. 
    Therefore, $b=b'_1\in S'$, so that $S\subseteq S'$. Similarly, $S'\subseteq S$, and $S=S'$.

    To prove that such a subset always exists, consider the set $S\subseteq \calA(\lambda,\mu)$ consisting of all elements $b$ with connected influence such that if $b=b_1b_2$ for $b_1$ and $b_2$ independent, then $b_1=1$ or $b_2=1$. 
    We claim that for any $\sigma\in \calA(\lambda,\mu)$, we have that $\sigma=b\sigma'$ where $b\in S$ and the elements $b$ and $\sigma'$ are independent. 
    Indeed, if $\sigma$ itself is not already an element of $S$, then $\sigma=\sigma_1\sigma_2$ with $\sigma_1<_R\sigma$ where the elements $\sigma_1$ and $\sigma_2$ are independent. If $\sigma_1\notin S$, then we can repeat this process. As the left factor decreases in the right weak order each time we repeat, eventually this process terminates with $\sigma = b\sigma'$ with $b\in S$ and the elements $b$ and $\sigma'$ independent. Iterating this process yields 
    \begin{align*}
\sigma=b_1b_2\cdots b_k,
    \end{align*} where the elements $b_i$ and $b_{i+1}b_{i+2}\cdots b_k$ are independent for all $1\leq i\leq k-1$. 
    By \Cref{lem:properties_of_influence} item (3), $I(b_j)\subseteq I(b_{i+1}b_{i+2}\cdots b_k)$ for $j>i$.
    Consequently, 
    \begin{align*}
        \overline{I}(b_i)\cap{I}(b_j)\subseteq \overline{I}(b_i)\cap I(b_{i+1}b_{i+2}\cdots b_k)=\emptyset.
    \end{align*} Hence, we have written $\sigma$ as a product of elements from a pairwise independent subset of $S$.

    For item (2): To prove that there is a bijection between pairwise independent subsets of $S$ and elements of $\calA(\lambda,\mu)$, we only need to demonstrate that the product of elements defining any pairwise independent subset of $S$ is in $\calA(\lambda,\mu)$ and that for any $\sigma\in \calA(\lambda,\mu)$, there exists a unique pairwise independent subset of $S$ for which $\sigma$ is a product of its elements.

    Let $B\subseteq S$ be a pairwise independent subset and let $\sigma=\prod_{b\in B}b$. Independent elements of $W$ necessarily commute, so this product is well-defined. For any $i, 1 \le i \le r$ (where $r$ is the rank of the Lie algebra), 
    the coefficient of $\alpha_i$ in $\sigma(\lambda+\rho)-\rho-\mu$
    is either: \begin{itemize}
        \item[(i)] the coefficient of $\alpha_i$ in
        $b(\lambda+\rho)-\rho-\mu$ for a unique $b\in B$ with $i\in I(b)$, or
        \item[(ii)] the coefficient of $\alpha_i$ in $\lambda-\mu$ if no such $b$ exists.
    \end{itemize}
    Since in (i), $b\in \calA(\lambda,\mu)$ and in (ii), $1\in \calA(\lambda,\mu)$, the coefficient of $\alpha_i$
    is nonnegative in either case. Hence, 
    $\sigma(\lambda+\rho)-\rho-\mu$ is a nonnegative linear combination of simple roots, and $\sigma\in \calA(\lambda,\mu)$.

    Suppose $\sigma=b_1b_2\cdots b_k=b'_1b'_2\cdots b'_{k'}$ where $\{b_1,b_2,\hdots,b_k\}$ and $\{b'_1,b'_2,\hdots,b'_{k'}\}$ are distinct, pairwise independent subsets of $S$. Fix $1\leq m\leq k$ and consider $i\in I(b_m)$. Necessarily, by \Cref{lem:properties_of_influence} item (3), $i\in I(b'_{m'})$ for exactly one $1\leq m'\leq k'$, hence $k\le k'$. Similarly, $k'\le k$, so $k=k'$. Since $I(b)$ is connected for all $b\in S$, we must have that $I(b_m)=I(b'_{m'})$. Without loss of generality, we can assume $m=m'$, so $I(b_m)=I(b'_m)$ for all $1\leq m\leq k$.
    
    Now for $j\in I(b_m)$, we have that \begin{align*}
        \sigma(\alpha_j)=b_1b_2\cdots b_k(\alpha_j)&=b_m(\alpha_j), \mbox{ and}\\
        \sigma(\alpha_j)=b'_1b'_2\cdots b'_{k}(\alpha_j)&=b'_m(\alpha_j).
    \end{align*}
    By \Cref{lem:properties_of_influence} item (2), $b_m=b'_m$. 
    Hence, the two products correspond to the same pairwise independent subset of~$S$.
\end{proof}

\begin{example}

As an example of Theorem \ref{thm:characterize_BAS}, we consider the case of $\lambda = \hroot, \mu = -\hroot$ for rank $r=4$. The set $\mathcal A(\lambda, \mu)$ consists of the following 11 elements of the Weyl group: \[1,~ s_1,~ s_2,~ s_3,~ s_4,~ s_1s_3,~ s_1s_4,~ s_2s_3,~ s_3s_2,~ s_2s_4,~ s_2s_3s_2.\] Each of these is a commuting product of the following 7 elements of the Weyl group:\[s_1,~ s_2,~ s_3,~ s_4,~ s_2s_3,~ s_3s_2,~ s_2s_3s_2,\] which comprise the set $\BAS(\lambda,\mu)$. Every element of the alternation set is a (possibly empty) commuting product of elements of $\BAS(\lambda,\mu)$, and each element of $\BAS(\lambda,\mu)$ has connected influence. We can draw a graph with vertex set given by these 7 elements with an edge drawn between them if they are not independent (see \Cref{fig:dependence_graph}). Since the nodes other than $s_1$ and $s_4$ in this graph form a 5-clique, any independent set contains at most one of these nodes. The total number of independent sets is thus 11: \begin{equation}
\emptyset,~ \{s_1\},~ \{s_4\},~ \{s_1, s_4\},~ \{s_2\},~ \{s_3\},~ \{s_2s_3\},~ \{s_3s_2\},~ \{s_2s_3s_2\},~ \{s_1, s_3\},~ \{s_2, s_4\}. \nonumber
\end{equation} If we multiply the elements within each independent set (all of which commute), we obtain each element of $\mathcal A(\hroot, -\hroot)$.

\end{example}

\begin{example}
      In type $A_8$, consider the dominant integral weights $\lambda$ and $\mu$ corresponding to the integer partitions $(4,3,1)$ and $(2,1^6)$ respectively. The set $\BAS(\lambda,\mu)$ consists of the following elements of $S_9$. \begin{center}
          \begin{tabular}{c|c}
              Length & Basic allowable subwords \\ \hline
              1 & $s_1$, $s_2$, $s_3$, $s_4$, $s_5$, $s_6$\\
              2 & $s_3s_4$, $s_4s_3$, $s_4s_5$, $s_5s_4$, $s_5s_6$\\
              3 & $s_3s_4s_3$, $s_3s_4s_5$, $s_3s_5s_4$, $s_4s_5s_4$, $s_4s_5s_6$\\
              4 & $s_3s_4s_5s_4$
          \end{tabular}
      \end{center}
\end{example}

\begin{example}
    In type $C_5$, consider the dominant integral weight \begin{align*}
        \lambda&=3\alpha_1+5\alpha_2+7\alpha_3+9\alpha_4+5\alpha_5,\text{ and the weight}\\
        \mu&=2\alpha_2+2\alpha_2+3\alpha_3+3\alpha_4+3\alpha_5.
    \end{align*} The set $\BAS(\lambda,\mu)$ consists of the following elements of the corresponding Weyl group.
    \begin{center}
          \begin{tabular}{c|c}
              Length & Basic allowable subwords \\ \hline
              1 & $s_2$, $s_3$, $s_4$, $s_5$\\
              2 & $s_2s_3$, $s_3s_2$, $s_3s_4$, $s_4s_3$, $s_4s_5$\\
              3 & $s_2s_3s_2$, $s_2s_4s_3$, $s_3s_4s_2$, $s_3s_4s_3$, $s_4s_3s_2$\\
              4 & $s_2s_4s_3s_2$, $s_3s_4s_3s_2$
          \end{tabular}
      \end{center}
\end{example}

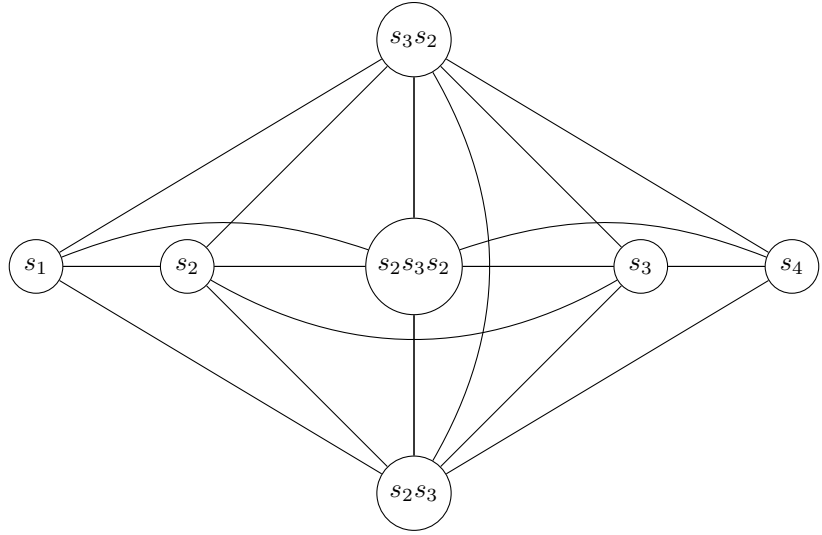
\begin{figure}
    \centering
\begin{tikzpicture}
\node at (0, 2) [circle,draw=black,fill=white] (1){$s_1$};
\node at (2, 2) [circle,draw=black,fill=white] (2){$s_2$};
\node at (8, 2) [circle,draw=black,fill=white] (3){$s_3$};
\node at (10, 2) [circle,draw=black,fill=white] (4){$s_4$};
\node at (5, 2)
[circle,draw=black,fill=white] (s){$s_2s_3s_2$};
\node at (5, -1)
[circle,draw=black,fill=white] (j){$s_2s_3$};
\node at (5, 5)
[circle,draw=black,fill=white] (m){$s_3s_2$};
\draw (1) -- (2) -- (s) -- (j) -- (1) -- (m) -- (s) -- (2) -- (1);
\draw (4) -- (3) -- (s) -- (j) -- (4) -- (m) -- (s) -- (3) -- (4);
\draw (2) -- (m) -- (3) -- (j) -- (2);
\draw[-] (s) to[bend right = 20] (1);
\draw[-] (s) to[bend left = 20] (4);
\draw[-] (m) to[bend left = 30] (j);
\draw[-] (3) to[bend left = 30] (2);
\end{tikzpicture}
    \caption{A graph on the elements of $\BAS(\hroot,-\hroot)$ in rank $r=4$ with edges drawn between nonindependent elements.}
    \label{fig:dependence_graph}
\end{figure}

\begin{remark}
    \Cref{thm:characterize_BAS} shows that $\calA(\lambda,\mu)$ has the structure of an \textit{independence system} or \textit{abstract simplicial complex} because each element of $\calA(\lambda,\mu)$ corresponds to a pairwise independent subset of $\BAS(\lambda,\mu)$ and the property of being pairwise independent is downward-closed.
\end{remark}

The following proof, providing sufficient conditions for a set of words to be the set $\BAS(\lambda,\mu)$, generalizes \cite[Theorem 4.5]{Lauren}.

\begin{proposition}\label{prop:BAS_properties}
    Let $\lambda$ be a dominant integral weight and let $\mu$ be a weight such that $\calA(\lambda,\mu)\neq\emptyset$.
    Let $S \subseteq \calA(\lambda, \mu)$, where $1 \notin S$, be such that
    \begin{itemize}
        \item[(C1)] any simple transposition $s_i\in \calA(\lambda,\mu)$ is contained in $S$,
        \item[(C2)] for each $\sigma \in S$, its influence $I(\sigma)$ is connected, and
        \item[(C3)] if $\sigma$ and $\tau$ in $S$ are not independent, the product $\sigma\tau$ falls into one of the following cases: \begin{enumerate}
        \item $\sigma\tau\in S$,
        \item $\sigma\tau=\nu_1\nu_2\cdots\nu_m$ where
        $\{\nu_1,\nu_2,\ldots,\nu_m\}\subseteq S$ is a pairwise independent set and \\
        $\ell(\nu_1)+\ell(\nu_2)+\cdots+\ell(\nu_m)<\ell(\sigma)+\ell(\tau)$, or
        \item $\sigma\tau$ contains a word not in $\calA(\lambda,\mu)$ as a consecutive subword.
    \end{enumerate}
    \end{itemize}
    Then, $\BAS(\lambda,\mu)=S$.
\end{proposition}

\begin{proof}
    By \Cref{thm:characterize_BAS}, it suffices to show that every element $\sigma\in \calA(\lambda,\mu)$ can be written as the product of a pairwise independent subset of $S$.

    Let $\sigma\in \calA(\lambda,\mu)$. By (C1) and \Cref{cor:subword_property}, any simple transposition appearing in a reduced expression for $\sigma$ is contained in $S$. Thus, we can write $\sigma$ as a product of elements of $S$. Choose a product $\sigma=b_1b_2\cdots b_k$ with $\ell(b_1)+\ell(b_2)+\cdots+\ell(b_k)$ minimal, and amongst these choose the product with $k$ minimal. Suppose there were two elements $b_i$ and $b_j$ in this product that are not independent. Without loss of generality, we can assume $b_i$ and $b_j$ are adjacent in the product (as an element could be swapped with any independent element adjacent to it). 
    By assumption, there are three possible cases. In case (1), we could replace $b_ib_j$ with a single element from $S$. Since $\ell(\sigma\tau)\leq\ell(\sigma)+\ell(\tau)$, this replacement must either reduce the sum of the lengths, contradicting its minimality or leave the sum of the lengths unchanged, contradicting the assumption that $k$ was minimal among products with that length sum. In case (2), we can replace $b_ib_j$ with a pairwise independent product decreasing the overall length sum, contradicting minimality of the length sum. In case (3), $b_ib_j$ contains a subword not in $\calA(\lambda,\mu)$. By \Cref{cor:subword_property}, this contradicts $\sigma\in \calA(\lambda,\mu)$. Hence, any pair $b_i, b_j$ in the product is independent, so $\{b_1,b_2,\ldots,b_k\}$ is a pairwise independent subset of $S$.
\end{proof}

Finally, in Proposition~\ref{prop:MovingToLargerWeights}, we relate the sets $\BAS(\lambda,\mu)$ and $\BAS(\lambda,\nu)$ for a dominant integral weight $\lambda$ and integral weights $\mu$ and $\nu$ satisfying $\nu\leq\mu$. To do so, we require the following lemma.

\begin{lemma}\label{lem:ContainmentsOfAltSets}
    Let $\lambda$, $\mu$, $\nu$ be integral weights. If $\nu\leq\mu$, then $\mathcal{A}(\lambda,\mu)\subseteq \mathcal{A}(\lambda,\nu)$.
\end{lemma}

\begin{proof}
    First, we claim that if $\alpha$ and $\beta$ are integral weights and $\alpha\leq\beta$, then $\wp(\alpha)>0$ implies that $\wp(\beta)>0$. Indeed, in this case, $\beta=\alpha+\sum_{i=1}^n c_i\alpha_i$ where $c_i\in\Z_{\geq0}$ for each $i$. Consequently, a vector partition of $\alpha$ can be completed to a vector partition of $\beta$ by adding $c_i$ parts equal to $\alpha_i$ for each $i$. 
    
    Now, suppose $\sigma\in\mathcal{A}(\lambda,\mu)$. Then $\wp(\sigma(\lambda+\rho)-\mu-\rho)>0$.
    Because $\nu\leq\mu$, we have that $\sigma(\lambda+\rho)-\mu-\rho\leq \sigma(\lambda+\rho)-\nu-\rho$. Hence, $\wp(\sigma(\lambda+\rho)-\nu-\rho)>0$ and $\sigma\in\mathcal{A}(\lambda,\nu)$, completing the proof.
\end{proof}

\begin{proposition}\label{prop:MovingToLargerWeights}
Let $\lambda$ be an integral dominant weight, and let $\mu$ and $\nu$ be two other integral weights such that $\nu\leq\mu$. 
Then, $\BAS(\lambda,\mu) = \BAS(\lambda,\nu) \cap \mathcal{A}(\lambda,\mu)$. 
\end{proposition}

\begin{proof}
Let $S = \BAS(\lambda,\nu) \cap \mathcal{A}(\lambda,\mu)$. We show $S = \BAS(\lambda,\mu)$ by showing $S$ satisfies the three properties in \Cref{prop:BAS_properties}.
\begin{enumerate}
    \item[(C1)]  If $s_i \in\mathcal{A}(\lambda,\mu)$, then 
    $s_i \in \mathcal{A}(\lambda,\nu)$
    by \Cref{lem:ContainmentsOfAltSets}.
    By \Cref{prop:BAS_properties}, $s_i \in \BAS(\lambda,\nu)$. Therefore, $s_i \in S$. 
    \item[(C2)] Since $S \subseteq \BAS(\lambda,\nu)$, we automatically know each $\sigma \in S$ has connected influence. 
    \item[(C3)] Let $\sigma,\tau$ be a pair of nonindependent elements of $S$. Since $\sigma,\tau \in \BAS(\lambda,\nu)$, \Cref{prop:BAS_properties} gives three possibilities for the product $\sigma\tau$.
    \begin{enumerate}[(1)]
        \item If $\sigma\tau \in \BAS(\lambda,\nu)$, then either $\sigma\tau \in S$ or $\sigma\tau \notin \mathcal{A}(\lambda,\mu)$. If we are in the latter case, then $\sigma\tau$ necessarily contains a consecutive subword that is not in $\mathcal{A}(\lambda,\mu)$. Thus, we are either in item (1) or item (3) from Condition (C3) of \Cref{prop:BAS_properties} with respect to $S$.
        \item If $\sigma\tau = \beta_1\beta_2 \cdots \beta_m$ where $\{\beta_1,\beta_2,\ldots,\beta_m\} \subseteq \BAS(\lambda,\nu)$ is a pairwise independent set and $\ell(\beta_1) + \ell(\beta_2) + \cdots + \ell(\beta_m) < \ell(\sigma) + \ell(\tau)$, then either each $\beta_i \in S$ or there exists $\beta_j \notin \mathcal{A}(\lambda,\mu)$.  In the latter case, $\beta_j$ contains a consecutive subword not in $\mathcal{A}(\lambda,\mu)$, hence $\sigma\tau$ does as well. Thus, we are either in item (2) or (3) from Condition (C3) of \Cref{prop:BAS_properties} with respect to $S$.
        \item If $\sigma \tau$ contains a consecutive subword which is not in $\mathcal{A}(\lambda,\nu)$, then
        $\sigma \tau$ also contains a consecutive subword which is not in $\mathcal{A}(\lambda,\mu)$ by \Cref{lem:ContainmentsOfAltSets}. Therefore, we are in item (3) from Condition (C3) of \Cref{prop:BAS_properties} with respect to $S$.
    \end{enumerate} 
    Therefore, the product of any pair of nonindependent elements $\sigma,\tau \in S$ fits into one of the three items of Condition (C3) of \Cref{prop:BAS_properties} with respect to $S$. \qedhere 
\end{enumerate}
\end{proof}

\section{The Lie algebra of type \texorpdfstring{$A$}{A}}\label{sec:NegAlpha}

In contrast with the type-agnostic perspective of the previous section, we focus in this section exclusively on the Lie algebra of type $A$.
Let $M_{r+1}(\C)$ denote the collection of $(r+1)\times (r+1)$ 
complex-valued matrices and recall that $\mathfrak{sl}_{r+1}(\C) = \{X \in M_{r+1}(\C): \text{Tr}(X) = 0\}$ is the Lie algebra of type $A_r$, with $r\geq 1$.
Recall that $\mathfrak{sl}_{r+1}(\C)$
is a complex vector space with a bilinear product, called the Lie bracket, defined by the commutator bracket $[X,Y]=XY-YX$ for $X,Y\in\mathfrak{sl}_{r+1}(\C)$, which is skew-symmetric and satisfies the Jacobi identity. 
Given $e_1,e_2,\ldots,e_{r+1}$, the standard basis vectors of $\R^{r+1}$, for each $i, 1 \le i \le r$, we define
$\alpha_i = e_i - e_{i+1}$.
In $\mathfrak{sl}_{r+1}(\C)$,
 the simple roots are $\Delta=\{\alpha_1,\alpha_2,\ldots,\alpha_r\}$, 
the positive roots are 
$\Phi^+=\Delta \cup \{\alpha_i+\alpha_{i+1}+\cdots+\alpha_j:1\leq i<j\leq r\}$,
and the negative roots are $-\Phi^+$. 
The \defn{highest root} is denoted $\hroot=\alpha_1+\alpha_2+\cdots+\alpha_r$ and the \defn{negative highest root} is denoted by $-\hroot=-(\alpha_1+\alpha_2+\cdots+\alpha_r)$.
We recall that the Weyl group $W$ of type $A_r$ is isomorphic to the symmetric group $\mathfrak{S}_{r+1}$.
Our main results characterize and enumerate the elements of the Weyl alternation sets $\calA_r(\hroot,\mu)$ when $\mu$ is a negative root of $\mathfrak{sl}_{r+1}(\mathbb{C})$. 
We introduce a subscript $r$ on the Weyl alternation set to specify the rank of the Lie algebra.

The following computations form the backbone of our proofs in this section. 

\begin{lemma}\label{lem:SmallerComputation}
For $\mathfrak{g}=\mathfrak{sl}_{r+1}(\C)$, the type $A_r$ simple Lie algebra,
\begin{enumerate}
\item $s_1(\tilde{\alpha}) = \tilde{\alpha} - \alpha_1$ and $s_r(\tilde{\alpha}) = \tilde{\alpha} - \alpha_r$,
\item for $2 \leq j \leq r-1$, $s_j(\tilde{\alpha}) = \tilde{\alpha}$, and
\item for all $1 \leq j \leq r$, $s_j(\rho) = \rho - \alpha_j$.
\end{enumerate}    
\end{lemma}

By applying this iteratively, we can identify the action of any element of the Weyl group on the highest root.
\begin{proof}
The computations are straightforward with (1) and (2) appearing in \cite[Lemma 3.1]{Lauren} and (3)  in \cite[Lemma 3.2]{Lauren}.
 \end{proof}

Now, we iteratively apply \Cref{lem:SmallerComputation} to the elements of $W$, the Weyl group of the Lie algebra of type $A_r$.

\begin{lemma}\label{lem:LargerComputations}
The action of Weyl group elements on $\hroot$ and $\rho$ satisfy the following.
\begin{enumerate}
    \item If $i_1, i_2,\ldots,i_k \in \{2,3,\ldots,r-1\}$, then $s_{i_k}s_{i_{k-1}}\cdots s_{i_1}(\tilde{\alpha}) = \tilde{\alpha}$.
    Otherwise, if $i_j$ is the first occurrence of $1$ or $r$ in the list $i_1,i_2,\ldots,i_k$, then $s_{i_k}s_{i_{k-1}} \cdots s_{i_1}(\tilde{\alpha}) = s_{i_k}s_{i_{k-1}} \cdots s_{i_j}(\tilde{\alpha})$.
    \item For $j < r$, $s_j s_{j-1} \cdots s_1(\tilde{\alpha}) = \tilde{\alpha} - \sum_{i=1}^j \alpha_i$, and for $j > 1$, $s_js_{j+1}\cdots s_r(\tilde{\alpha}) = \tilde{\alpha} - \sum_{i=j}^{r} \alpha_i$.
\end{enumerate}
\begin{enumerate}
    \item[(3)] For $1\leq i\leq r$ and $i+k\leq r$, we have $s_{i}s_{i+1} \cdots s_{i+k}(\rho) = \rho - \sum_{j=i}^{i+k} (i+k+1-j)  \alpha_j $ and similarly $s_{i+k} \cdots s_{i+1} s_{i}(\rho) = \rho - \sum_{j = i}^{i+k} (j-i+1)\alpha_j $.
    \item[(4)] For $1\leq k\leq r-1$, 
    $s_ks_{k+1}(\rho) = \rho - 2 \alpha_k - \alpha_{k+1}$ and $s_{k+1}s_k(\rho) = \rho -  \alpha_k - 2 \alpha_{k+1}$.
    \item[(5)] For $1\leq k\leq r-1$, 
    $s_ks_{k+1}s_k(\rho) = \rho - 2 \alpha_k - 2\alpha_{k+1}$.
    \item[(6)] For $1\leq k\leq r-2$, 
    $s_{k+2}s_ks_{k+1}(\rho) = \rho - 2 \alpha_k - \alpha_{k+1} - 2\alpha_{k+2}$ and $s_{k+1}s_{k+2}s_k(\rho) = \rho - \alpha_k - 3 \alpha_{k+1} - \alpha_{k+2}$.
\end{enumerate}
\end{lemma}

\begin{proof}
Item (1) follows directly from \Cref{lem:SmallerComputation}. 

We show the first statement of (2) by induction on $j$. 
The case when $j = 1$ follows immediately from \Cref{lem:SmallerComputation}. 
Suppose the claim holds for $j-1$. Applying our inductive hypothesis, we find that  
\[s_j s_{j-1} \cdots s_1\left(\tilde{\alpha}\right) = s_j\left(\tilde{\alpha} - \sum_{i=1}^{j-1} \alpha_i\right) =s_j\left(\hroot - \alpha_{j-1}-\sum_{i=1}^{j-2}\alpha_i\right) = \tilde{\alpha} - \left(\alpha_{j-1} + \alpha_j\right) - \sum_{i=1}^{j-2} \alpha_i.
\]

The second statement follows analogously. 

Item (3) follows from an inductive argument similar to that used in item (2), and item (4) is a special case of (3). 
Items (5) and (6) are straightforward calculations using \Cref{lem:SmallerComputation}. 
\end{proof}

\subsection{Allowable subwords}\label{sec:allowable_subwords}

In this section, we introduce a set of elements of the type $A_r$ Weyl group and verify that each is in $\calA(\tilde\alpha,-\tilde\alpha)$. We show in \Cref{sec:BAS_for_negAlpha} that these elements are precisely $\BAS(\tilde\alpha,-\tilde\alpha)$.

\begin{proposition}\label{prop:mu=-hroot}
If $\sigma$ is one of the Weyl group elements
\begin{enumerate}
\item[(a)] $s_k$ with $1 \le k \le r$,
\item[(b)] $s_{k+1}s_{k}$ with $2 \le k \le r-2$,
\item[(c)] $s_{k}s_{k+1}$ with $2 \le k \le r-2$, 
\item[(d)] $s_{k}s_{k+1}s_{k}$ with $2 \le k \le r-2$, or
\item[(e)] $s_{k+2}s_{k}s_{k+1}$ with $2 \le k \le r-3$, 
\end{enumerate}
then $\sigma \in \calA(\tilde{\alpha},-\tilde{\alpha})$. 
\end{proposition} 

\begin{proof}
It suffices to check that for any $w \in W$ from the list, if $w(\tilde{\alpha} + \rho) = \tilde{\alpha} + \rho - \sum_{i=1}^r c_i \alpha_i$, then $c_i \leq 2$ for all $i$. 
First, if $w$ does not have a subword $s_1$ or $s_r$, then $w(\tilde{\alpha}) = \tilde{\alpha}$ by \Cref{lem:LargerComputations}~(1). 
Consequently, when $w$ does not have a subword $s_1$ or $s_r$, we need only focus on $w(\rho)$, and the desired implication follows by \Cref{lem:LargerComputations} parts (3)-(6). The only cases where we consider $s_1$ or $s_r$ is when $w$ is the simple reflection $s_1$ or $s_r$.
In these cases, the proof uses the simpler computations in \Cref{lem:SmallerComputation}.
Thus, in all cases, $w(\tilde{\alpha}+\rho)+\tilde{\alpha}-\rho=2\tilde{\alpha}-\sum_{i=1}^r c_i \alpha_i$ with all $c_i\leq 2$. As noted above, the result follows.
\end{proof}

\subsection{Forbidden subwords}

In this section, we collect some subwords that cannot appear as consecutive subwords in a reduced expression for an element of $\calA(\tilde\alpha,-\tilde\alpha)$.

\begin{lemma}\label{lem:forbidden words}
The words 
\[s_2s_1,\qquad s_1s_2,\qquad s_{r-1}s_r, \qquad \text{and} \qquad s_rs_{r-1},\] and for any $2 \le i \le r-1$, the words
\[s_{i-1}s_is_{i+1}, \qquad s_{i}s_{i-1}s_{i+1}, \qquad \text{and} \qquad s_{i+1}s_is_{i-1}\]
are not contained in $\calA(\hroot,-\hroot)$.
\end{lemma}

\begin{proof}

To show that the listed words are not in $\calA(\hroot,-\hroot)$, it suffices to check that $w(\tilde{\alpha}+\rho)+\tilde{\alpha}-\rho$ has a negative coefficient when expanded as a linear combination of simple roots, meaning that $\wp(w(\tilde{\alpha}+\rho)+\tilde{\alpha}-\rho)$ is necessarily zero. We denote the leftmost reflection in each word $w$ by $s_j$. Using \Cref{lem:SmallerComputation} and \Cref{lem:LargerComputations}, it is straightforward to compute that for any word $w$ listed, 
\begin{align}
w(\tilde{\alpha}+\rho)+\tilde{\alpha}-\rho=2\tilde{\alpha}-3\alpha_j-\sum_{i=1}^{j-1} c_i \alpha_i-\sum_{i=j+1}^{r} c_i\label{eq:comp} \alpha_i\end{align}
with $c_i\leq 2$ for all $i$. 
Observe that in \Cref{eq:comp} the coefficient of $\alpha_j$ is $-1$.
As noted above, the result follows. 
\qedhere
\end{proof}

\begin{lemma} \label{Highest:1234}
Let $1 \leq k \leq r-3$. The product of the four simple reflections $s_k$, $s_{k+1}$, $s_{k+2}$, and $s_{k+3}$ in any order is not contained in $\mathcal{A}(\hroot,-\hroot)$.    
\end{lemma}

\begin{proof}
Let $\sigma$ be a product of $s_k$, $s_{k+1}$, $s_{k+2}$, and $s_{k+3}$ in some order and suppose that $\sigma\in\mathcal{A}(\hroot,-\hroot)$. Because $s_{k+3}$ commutes with $s_k$ and $s_{k+1}$, we can move $s_{k+3}$ to the beginning or the end of the reduced expression for $\sigma$. That is, $\sigma=s_{k+3}\tau$ or $\sigma=\tau s_{k+3}$ where $\tau$ is a product of $s_k$, $s_{k+1}$, and $s_{k+2}$ in some order. By \Cref{lem:forbidden words}, we must have that $\tau=s_{k}s_{k+2}s_{k+1}$ for $2\leq k\leq r-3$. Then either \begin{align*}
    \sigma&=s_{k+3}\tau\\
    &=s_{k+3}s_{k}s_{k+2}s_{k+1}\\
    &=s_k(s_{k+3}s_{k+2}s_{k+1}),
\end{align*} containing a forbidden $s_{k+3}s_{k+2}s_{k+1}$ or \begin{align*}
    \sigma&=\tau s_{k+3}\\
    &=s_{k}s_{k+2}s_{k+1}s_{k+3}\\
    &=s_k(s_{k+2}s_{k+1}s_{k+3}),
\end{align*} containing a forbidden $s_{k+2}s_{k+1}s_{k+3}$. This contradicts the fact that $\sigma\in \mathcal{A}(\hroot,-\hroot)$.
\end{proof}

\subsection{Weyl alternation set \texorpdfstring{$\calA(\hroot,\mu)$}{A(alpha tilde,mu)} for \texorpdfstring{$\mu$}{mu} the negative highest root}\label{sec:BAS_for_negAlpha} In this section, we apply \Cref{prop:BAS_properties} to prove that the set of words given in \Cref{sec:allowable_subwords} is precisely the set $\BAS(\tilde\alpha,-\tilde\alpha)$. The following result is analogous to \cite[Lemma 4.3]{Lauren}, and we modify it to fit our purposes.

\begin{lemma}\label{lem:ProductOfTwoBASs}
    Let $S$ be the set of elements listed in \Cref{prop:mu=-hroot}.
    Let $\sigma, \tau\in S$ be nonindependent elements.
    Then, the product $\sigma\tau$ falls into one of the following three cases:
    \begin{enumerate}
        \item $\sigma\tau\in S$,
        \item $\sigma\tau=\nu_1\nu_2\cdots\nu_m$ where $\{\nu_1,\nu_2,\ldots,\nu_m\}$ is a (possibly empty) pairwise independent subset of $S$ and \\
        $\ell(\nu_1)+\cdots+\ell(\nu_m)<\ell(\sigma)+\ell(\tau)$, or
        \item $\sigma\tau$ contains a forbidden subword listed in \Cref{lem:forbidden words} or~\Cref{Highest:1234}.
    \end{enumerate}
\end{lemma}
\begin{proof}[Proof idea]
  We proceed via cases. Recall the possible cases for $\sigma$ and $\tau$ as in \Cref{prop:mu=-hroot}:
  \begin{enumerate}
        \item[(a)] $s_k$ with $1 \le k \le r$,
        \item[(b)] $s_{k+1}s_{k}$ with $2 \le k \le r-2$,
        \item[(c)] $s_{k}s_{k+1}$ with $2 \le k \le r-2$, 
        \item[(d)] $s_{k}s_{k+1}s_{k}$ with $2 \le k \le r-2$, or
        \item[(e)] $s_{k+2}s_{k}s_{k+1}$ with $2 \le k \le r-3$.
    \end{enumerate}
    
       For each pair of forms that $\sigma, \tau\in S$ can respectively take, we fix the indices of $\sigma$, then consider the range of indices of $\tau$ for which $\sigma$ and $\tau$ are not independent.
  
  \smallskip

  \noindent \textbf{Products of the form (a)(a):}
    Let $\sigma = s_k$ for some $1 \leq k \leq r$. For $\tau = s_{j}$ with $1\leq j\leq r$ not to be independent with $\sigma$, we must have $j\in\{k-1,k,k+1\}$.
    \begin{itemize}
        \item Let $j=k-1$ for $2\leq k\leq r$. Then $\sigma\tau=s_ks_{k-1}$. If $k=2$ (resp., $k=r$), then $s_2s_1$ (resp., $s_rs_{r-1}$) is forbidden by \Cref{lem:forbidden words}. That is, $\sigma\tau$ falls into case (3). Otherwise, $3\leq k\leq r-1$ and $s_ks_{k-1}\in S$. That is, $\sigma\tau$ falls into case (1). 
        \item Let $j=k$ for $1\leq k\leq r$. Then $\sigma\tau=s_ks_k=1$, a product of an empty subset of $S$, with $0=\ell(1)=\ell(s_ks_{k})<\ell(s_k)+\ell(s_{k})=2$. That is, $\sigma\tau$ falls into case (2).
        \item Let $j=k+1$ for $1\leq k\leq r-1$. Then $\sigma\tau=s_ks_{k+1}$. If $k=1$ (resp., $k=r-1$), then $s_1s_2$ (resp., $s_{r-1}s_r$) is forbidden by \Cref{lem:forbidden words}. That is, $\sigma\tau$ falls into case (3). Otherwise, $2\leq k\leq r-2$ and $s_ks_{k+1}\in S$. That is, $\sigma\tau$ falls into case (1).
    \end{itemize}

\noindent The remaining cases are verified similarly.
For the complete proof, see  \Cref{app:proofs}. 
\end{proof}

\smallskip

Next, we characterize the set of basic allowable subwords in the case $\mu=-\hroot$.

\begin{theorem} \label{thm:BAS_for_negAlpha}
    The set $\BAS(\tilde\alpha,-\tilde\alpha)$ of basic allowable subwords of $\calA(\tilde\alpha,-\tilde\alpha)$ consists of \begin{enumerate}
        \item[(a)] $s_k$ with $1 \le k \le r$,
\item[(b)] $s_{k+1}s_{k}$ with $2 \le k \le r-2$,
\item[(c)] $s_{k}s_{k+1}$ with $2 \le k \le r-2$, 
\item[(d)] $s_{k}s_{k+1}s_{k}$ with $2 \le k \le r-2$, and
\item[(e)] $s_{k+2}s_{k}s_{k+1}$ with $2 \le k \le r-3$.
\end{enumerate}
\end{theorem}

\begin{proof}
    Let $S$ denote the set of elements given in the statement of \Cref{thm:BAS_for_negAlpha}. We proceed by showing that $S$ satisfies the conditions in \Cref{prop:BAS_properties}. 

    Since $S$ contains all simple transpositions $s_k$ for $1\leq k\leq r$, it clearly satisfies (C1). Next, considering \Cref{lem:properties_of_influence}~(1), we can read off the influence of each element as the indices present on simple transpositions in their reduced expressions given above. As these sets are connected, $S$ satisfies (C2). Finally, \Cref{lem:ProductOfTwoBASs} establishes that $S$ satisfies (C3). Thus, $S=\BAS(\tilde\alpha,-\tilde\alpha)$, as desired.
\end{proof}

\subsection{Weyl alternation sets \texorpdfstring{$\calA(\hroot,\mu)$}{A(alpha tilde, mu)} for other negative roots \texorpdfstring{$\mu$}{mu}}\label{sec:OtherMu}

In this subsection, we consider other negative roots besides the negative highest root and apply \Cref{prop:MovingToLargerWeights} to obtain the Weyl alternation sets $\calA(\hroot,\mu)$ from that of $\calA(\hroot,-\tilde{\alpha})$ found above. Each $\mu\in\Phi^-$ can be written as \[\mu=-(\alpha_i+\alpha_{i+1}+\dots+\alpha_j)\] for some $1\leq i\leq j\leq r$.

For the remainder of this paper, we let
\[\alpha_{i,j} = \sum_{k=i}^j \alpha_k\] for any $1 \le i \leq j \le r$ and as expected $\alpha_{i,i}=\alpha_i$.
Because the actions of $s_1$ and $s_r$ on weights are distinct from that of $s_k$, with $2 \leq k \leq r-1$, we group the negative roots into three cases: 
(1) $1=i \leq j < r$; 
(2) $1<i\leq j=r$; and 
(3) $1<i\leq j<r$. 
In each case, the corresponding corollary below of \Cref{prop:MovingToLargerWeights} is obtained by computing the intersection of $\BAS(\hroot,-\hroot)$, as characterized in \Cref{thm:BAS_for_negAlpha}, with the Weyl alternation set for the relevant negative root.

\begin{example}
\label{ex:bas_survival}
    For $r=5$, one can readily verify that $s_4s_3$ is a basic allowable subword in $\BAS(\hroot,-\hroot)$. For $\mu=-\alpha_{1,3}$, note that \begin{align*}
        s_4s_3(\hroot+\rho)-\mu-\rho&=s_4(\hroot+\rho-\alpha_3)-\mu-\rho\\
        &=\hroot+\rho-\alpha_4-\alpha_3-\alpha_4-\mu-\rho\\
        &=2\alpha_1+2\alpha_2+\alpha_3-\alpha_4+\alpha_5
\end{align*} 
 has a negative coefficient on $\alpha_4$, hence $s_4s_3\notin\BAS(\hroot,\mu)$.
\end{example}

The following three corollaries are obtained by computing the intersections $\BAS(\hroot,-\hroot)\cap\calA(\hroot,\mu)$ for negative roots $\mu$. As illustrated in \Cref{ex:bas_survival}, only a subset of the basic allowable subwords survive.

\begin{corollary}\label{cor:BAS1j}
Let $\mu = -\alpha_{1,j}$ for $1 \leq j < r$.
The set $\BAS(\hroot, \mu)$ of basic allowable subwords of $\mathcal{A}(\hroot, \mu)$ consists of 
\begin{enumerate}
\noindent \begin{minipage}{0.45\textwidth}
    \item[(a)] $s_k$ with $1 \leq k < r$,
    \item[(b)] $s_{k+1}s_k$ with $2 \leq k \leq j-1$,
    \item[(c)] $s_ks_{k+1}$ with $2 \leq k \leq \min(j,r-2)$,
\end{minipage}
\begin{minipage}{0.45\textwidth}
    \item[(d)] $s_ks_{k+1}s_k$ with $2 \leq k \leq j-1$, and 
    \item[(e)] $s_{k+2}s_ks_{k+1}$ with $2 \leq k \leq j-2$.
\end{minipage}
\end{enumerate}
\end{corollary}

\begin{corollary}\label{cor:BASir}
Let $\mu = -\alpha_{i,r}$ for $1 < i \leq r$.
The set $\BAS(\hroot, \mu)$ of basic allowable subwords of $\mathcal{A}(\hroot, \mu)$ consists of
\begin{enumerate}
\noindent \begin{minipage}{0.45\textwidth}
    \item[(a)] $s_k$ with $1 < k \leq r$,
    \item[(b)] $s_{k+1}s_k$ with $\max(i-1,2) \leq k \leq r-2$,
    \item[(c)] $s_ks_{k+1}$ with $i \leq k \leq r-2$,
\end{minipage}
\begin{minipage}{0.45\textwidth}
    \item[(d)] $s_ks_{k+1}s_k$ with $i \leq k \leq r-2$, and 
    \item[(e)] $s_{k+2}s_ks_{k+1}$ with $i \leq k \leq r-3$.
\end{minipage}
\end{enumerate}
\end{corollary}

\begin{corollary}\label{cor:BASij}
Let $\mu = -\alpha_{i,j}$ for $1 < i \leq j < r$.
The set $\BAS(\hroot, \mu)$ of basic allowable subwords of $\mathcal{A}(\hroot, \mu)$ consists of 
\begin{enumerate}
\noindent \begin{minipage}{0.45\textwidth}
    \item[(a)] $s_k$ with $1 < k < r$,
    \item[(b)] $s_{k+1}s_k$ with $\max(i-1,2) \leq k \leq j-1$,
    \item[(c)] $s_ks_{k+1}$ with $i \leq k \leq \min(j,r-2)$,
\end{minipage}
\begin{minipage}{0.45\textwidth}
    \item[(d)] $s_ks_{k+1}s_k$ with $i \leq k \leq j-1$, and 
    \item[(e)] $s_{k+2}s_ks_{k+1}$ with $i \leq k \leq j-2$.
\end{minipage}
\end{enumerate}
\end{corollary}

\section{Enumerative results}\label{sec:enumerative}

 Having characterized the basic allowable words for Weyl alternation sets $\calA_r(
 \hroot,\mu
 )$ for negative roots $\mu$ in the Lie algebra of type $A_r$, we now provide some enumerative results and generating functions for these sets.
 To begin, we define the Fibonacci numbers by $F_0 = 0$, $F_1 = 1$, and for $n \geq 2$, $F_n = F_{n-1} + F_{n-2}$. 
Harris's work \cite[Theorem 1.2]{PHThesisPublication} establishes that $\vert \mathcal{A}_r(\hroot,0) \vert = F_r$. 
 We can recover this result from our study of $\mathcal{A}(\hroot,-\hroot)$ and \Cref{prop:MovingToLargerWeights}. 
  More precisely, the only elements from $\BAS(\hroot,-\hroot)$ that would also be in $\mathcal{A}(\hroot,0)$ are commuting products of $\{s_i: 2 \leq i \leq r-1\}$, the number of which is a Fibonacci number. 

We begin this section by enumerating families of Weyl alternation sets $\mathcal{A}_r(\hroot, \mu)$ where $\mu$ is the sum of only a few negative simple roots.
Throughout this section, we let $\BAS_r(\lambda,\mu)$ refer to the basic allowable subwords of $\mathcal{A}_r(\lambda,\mu)$. To prove such results, we require the following lemma. We point the reader to \cite[Proposition 2.3]{Harry} for a proof of the classical combinatorial result.

\begin{lemma}\label{lem:CountFibThings}
The number of subsets of $\{1,\ldots,n\}$ that do not contain a pair of consecutive numbers is $F_{n+2}$.
\end{lemma}

Our first result considers the case where $\mu$ is a negative root of height $1$, i.e., $\mu = -\alpha_i$ for some index $i$. 

\begin{proposition}\label{prop:EnumerateNegativeSimple}
 Let $r \geq 1$ and fix $1 \leq i \leq r$.
  \begin{enumerate}
     \item If $i = 1$ or $i = r$, then $\vert \mathcal{A}_r(\hroot,-\alpha_i) \vert = F_{r+1}$.
     \item 
     If $r > 2$ and $2 \leq i \leq r-1$, then  $\vert \mathcal{A}_r(\hroot,-\alpha_i) \vert = F_{r} + F_{i-1}F_{r-i-1} + F_{i-2}F_{r-i}$.
 \end{enumerate}
\end{proposition}
\begin{proof} For item (1):
Suppose $i = 1$. Then, \Cref{cor:BAS1j} yields  $\BAS_r(\hroot,-\alpha_1) = \{s_k : 1 \leq k < r\}$, so that $\mathcal{A}_r(\hroot,-\alpha_1)$ consists of commuting products of $r-1$ simple transpositions. 
{Since $s_i$ and $s_j$ commute if and only if $i$ and $j$ are nonconsecutive, this exactly corresponds to subsets of $\{1,\ldots,r-1\}$ that do not contain a pair of consecutive numbers.} By \Cref{lem:CountFibThings}, the number of such sets is $F_{r+1}$.  The proof for $i=r$ is similar.

 For item (2): First, suppose $i = 2$.  By \Cref{cor:BASij}, $\BAS_r(\hroot,-\alpha_2) = \{s_k : 1 < k < r\} \cup \{s_{2}s_{3}\}$. If $r > 3$, the elements of $\mathcal{A}_r(\hroot,-\alpha_2)$ can be split into two sets: those that contain $s_2s_3$ as a subword and those that do not. The latter can be identified with $\mathcal{A}_r(\hroot,0)$ so that there are $F_{r}$ such elements. The former are in bijection with commuting products of $\{s_k: 5 \le k < r\}$, of which there are $F_{r-3}$ considering \Cref{lem:CountFibThings}. If $r = 3$, then we only have the former case; since we defined $F_{-1} = 0$, the formula still holds. 
 Thus,
 \[\vert\mathcal{A}_r(\hroot,-\alpha_2)\vert=F_r+F_{r-3}=F_r+F_1F_{r-3}+F_0F_{r-2}\]
 as claimed. 
 The case $i = r-1$ is analogous.

Next, suppose $2 < i < r-1$. By \Cref{cor:BASij}, $\BAS_r(\hroot,-\alpha_i) = \{s_k : 1 < k < r\} \cup \{s_{i}s_{i-1}, s_is_{i+1}\}$. 
We now can break $\mathcal{A}_r(\hroot,-\alpha_i)$ into three disjoint sets:
\begin{enumerate}
    \item those elements without subwords $s_is_{i-1}$ and $s_{i}s_{i+1}$, 
    \item those with subword $s_is_{i+1}$ and without $s_{i}s_{i-1}$, and
    \item those with subword $s_{i}s_{i-1}$ and without $s_{i}s_{i+1}$.
    \end{enumerate}
Since elements in the Weyl alternation set are products of pairwise independent subsets of the basic allowable subwords, we have no elements in $\mathcal{A}_r(\hroot,-\alpha_i)$ with both $s_is_{i+1}$ and $s_{i}s_{i-1}$ as subwords. 

The number of elements described by (1) is again $F_{r}$. 
The number of elements described by (2) is equal to the number of pairwise commuting products of $\{s_j: 2 \leq j \leq i-2\}$ and $\{s_k: i+3 \leq k \leq r-1\}$, of which there are $F_{i-1}F_{r-i-1}$ considering \Cref{lem:CountFibThings}.

The number of elements (3) is equal to the number of pairwise commuting products of $\{s_j: 2 \leq j \leq i-3\}$ and $\{s_k:  i+2 \leq j \leq r-1\}$, of which there are $F_{i-2}F_{r-i}$ considering \Cref{lem:CountFibThings}.
Summing the contributions of these sets gives us
\[\vert\mathcal{A}_r(\hroot,-\alpha_i)\vert=F_r+F_{i-1}F_{r-i-1}+F_{i-2}F_{r-i},\] as desired.
\end{proof}

If we set $i = 3$ in item (2) of \Cref{prop:EnumerateNegativeSimple}, we recover the Lucas numbers.

\begin{corollary}
Define the Lucas numbers by $L_0 = 2$, $L_1 = 1$ and $L_n = L_{n-1} +L_{n-2}$ for all $n \geq 2$. Then, if $r \geq 4$, we have $\vert \mathcal{A}_r(\hroot, -\alpha_3) \vert = L_{r-1}$.
\end{corollary}
Next, we consider the case where $\mu$ is a negative root of height 2. As arguments for the cardinalities of subsets of $\mathcal{A}_r(\hroot,-\alpha_i - \alpha_{i+1})$ considered in the proof of the following result are similar to those given in the proof \Cref{prop:EnumerateNegativeSimple}, for the sake of brevity, such arguments are omitted.

\begin{proposition} $ $
\begin{enumerate}
    \item If $r\geq 2$ and $i = 1$ or $i = r-1$, then $\vert \mathcal{A}_r(\hroot,-\alpha_i - \alpha_{i+1})\vert= F_{r+1} + F_{r-3} = 3 F_{r-1}$.
    \item If $r\geq 4$ and $2 \leq i \leq r-2$, then $\vert \mathcal{A}_r(\hroot,-\alpha_i - \alpha_{i+1})\vert = F_{r} + F_{i-2}F_{r-i} + 3 F_{i-1} F_{r-i-1} + F_{i} F_{r-i-2}$. 
\end{enumerate}
\end{proposition}

\begin{proof}
For item (1): 
Suppose $r\ge 2$ and $i = 1$. 
\begin{itemize}
\item If $r=2$, then $-\alpha_1-\alpha_2=-\hroot$.
By 
\Cref{thm:BAS_for_negAlpha},
we have that $\BAS_2(\hroot, -\hroot)=\{s_1,s_2\}$. 
Hence,
$|\calA_2(\hroot,-\hroot)|=|\{1,s_1,s_2\}|=3=3F_{1}$.

\item If $r=3$, then by 
\Cref{cor:BAS1j},
we have that $\BAS_3(\hroot, -\alpha_1-\alpha_2)=\{s_1,s_2\}$. Hence,
$|\calA_3(\hroot,-\alpha_1-\alpha_2)|=|\{1,s_1,s_2\}|=3=3F_{2}$.
\item Now for $r\geq 4$, using \Cref{cor:BAS1j}, we have that $\BAS_r(\hroot,-\alpha_1-\alpha_2) = \{s_k : 1 \leq k < r\} \cup \{s_2s_3\}$. 
Hence, $\mathcal{A}_r(\hroot,-\alpha_1-\alpha_2)$ can be split into two sets of elements: those that contain $s_2s_3$ as a subword and those that do not.
The number of those elements not containing $s_2s_3$ is $F_{r+1}$ because we are choosing nonconsecutive indices from the set $\{1,2,\ldots,r-1\}$.
The number of elements containing $s_2s_3$ is given by $F_{r-3}$. Hence $|\calA_r(\hroot,-\alpha_1-\alpha_2)|=F_{r+1}+F_{r-3}=3F_{r-1}$.
\end{itemize}
When $i=r-1$, the argument is analogous.

\vspace{.1in}

For item (2): Suppose $r\geq 4$ and $2 \leq i \leq r-2$.
\begin{itemize}
 \item If $2 <i < r-2$, then by \Cref{cor:BASij}, we have that \[\BAS(\hroot, -\alpha_i - \alpha_{i+1}) = \{s_k : 1 < k < r\} \cup \left\{\begin{array}{l}
      \tau_1=s_is_{i-1},\\
      \tau_2=s_{i+1}s_i,\\
      \tau_3=s_is_{i+1},\\
      \tau_4=s_{i+1}s_{i+2},\\
      \tau_5=s_is_{i+1}s_i
 \end{array}\right\}.\]
 We can partition $\mathcal{A}_r(\hroot,-\alpha_i)$ into the following disjoint sets:
\begin{enumerate}
    \item those elements containing $\tau_1$,
    \item those elements containing $\tau_2$,
    \item those elements containing $\tau_3$,
    \item those elements containing $\tau_4$,
    \item those elements containing $\tau_5$, and
    \item those elements without any $\tau_n$ for $1\leq n\leq 5$. 
\end{enumerate}
Observe that whenever an element contains $\tau_n$ with $n\in\{1,2,3,4,5\}$, it cannot contain $\tau_m$ with $m \neq n$ as no two elements $\tau_n$ are independent. 
 
The sets of elements considered in cases (2), (3), and (5) have the same cardinality. 
In particular, the elements in each case are in bijection with the number of ways to select nonconsecutive indices from $\{2,3,\ldots,i-2\} \cup \{i+3,i+4,\ldots,r-1\}$, so that {there are} $F_{i-1}F_{r-i-1}$ such elements.
Hence, we have in total over these three cases $3F_{i-1}F_{r-i-1}$ elements in the Weyl alternation set.
Similarly, the number of elements considered in case (1) is given by $F_{i-2}F_{r-i}$; the number of elements considered in case (4) is given by $F_{i}F_{r-i-2}$; and,
the number of elements considered in (6) is given by $F_{r}$.

Therefore,
\[|\calA_r(\hroot,-\alpha_i-\alpha_{i+1})|=3F_{i-1}F_{r-i-1}+F_{i-2}F_{r-i}+F_{i}F_{r-i-2}+F_{r}.\]
\item If $i = 2$, then the same argument holds as in the previous item, except we do not have a term $s_{i-1}s_i$ since $s_1$ is not in the support of any element of $\mathcal{A}(\hroot, -\alpha_2 -\alpha_3)$. Since $F_{0} = 0$,
the formula still holds.

\item If $i = r-2$, then the argument is analogous to the case where $i=2$. In this case, there is no $s_{i+1}s_{i+2}$ term since $s_r$ is not in the support of any element of $\mathcal A(\hroot, -\alpha_{r-2}-\alpha_{r-1})$. Since $F_{r-i-2}=F_0=0$, the formula still holds. \qedhere
\end{itemize}
\end{proof}

 The prevalence of sums of Fibonacci numbers in the previous results comes from a general behavior concerning cardinalities of these Weyl alternation sets.
  Henceforth, we fix a pair $1 \leq i \leq j$, let $\mu=-(\alpha_i + \alpha_{i+1} + \cdots + \alpha_j)$, vary $r$, and compute $\mathcal{A}_r(\hroot, \mu)$. 
 Naturally, this expression only makes sense when $r \geq j$. Next, we show that the cardinality of $\mathcal{A}_r(\hroot,\mu)$ satisfies the Fibonacci recurrence.

\begin{proposition}\label{prop:RecurrenceOnSizesFixedRoot}
Let $r\geq 3$ and 
 $1 \leq i \leq j \leq r-2 $. If $\mu = -\alpha_{i,j}$, 
then
\[
\vert \mathcal{A}_r(\hroot, \mu) \vert = \vert \mathcal{A}_{r-1}(\hroot, \mu) \vert  + \vert \mathcal{A}_{r-2}(\hroot, \mu) \vert. 
\]
\end{proposition}

\begin{proof}
First, suppose $j < r-2$. From \Cref{cor:BAS1j} and \Cref{cor:BASij}, we get that $\BAS_r(\hroot,\mu) = \BAS_{r-1}(\hroot,\mu) \cup \{s_{r-1}\}$, so we can partition the elements of $\mathcal{A}_r(\hroot,\mu)$ into two sets based on whether the elements contain $s_{r-1}$ as a subword or not. 
The set with the elements not containing $s_{r-1}$ is equivalent to $\mathcal{A}_{r-1}(\hroot,\mu)$. The set of elements containing $s_{r-1}$ are in bijection with $\mathcal{A}_{r-2}(\hroot,\mu)$. 
Thus, $|\calA_r(\hroot,\mu)|=|\calA_{r-1}(\hroot,\mu)|+|\calA_{r-2}(\hroot,\mu)|$, as claimed.  

Now, let $j = r-2$. In this case, there are two elements of $\BAS_r(\hroot,\mu)$ containing a $s_{r-1}$ in their reduced expressions: $s_{r-1}$ and $s_{r-2}s_{r-1}$. 
The elements of $\mathcal{A}_r(\hroot,\mu)$ not including either of these as a subword can again be regarded as elements of $\mathcal{A}_{r-1}(\hroot,\mu)$.
Hence, there are $|\mathcal{A}_{r-1}(\hroot,\mu)|$ such elements. 
Now, consider the subset of $\mathcal{A}_r(\hroot,\mu)$ consisting of elements containing the basic allowable subword $s_{r-1}$. 
These elements are in bijection with elements of $\mathcal{A}_{r-2}(\hroot,\mu)$ that do not include the basic allowable subword $s_{r-2}$ via the map which removes $s_{r-1}$.
Similarly, the subset of elements in $\mathcal{A}_r(\hroot,-\mu)$ consisting of elements containing the basic allowable subword $s_{r-2}s_{r-1}$ is in bijection with elements of $\mathcal{A}_{r-2}(\hroot,-\mu)$ that include the basic allowable subword $s_{r-2}$. 
As elements in $\calA_{r-2}(\hroot,\mu)$ either contain $s_{r-2}$ or do not, the total number of elements is given by $|\calA_{r-2}(\hroot,\mu)|$. Thus, the total number of elements of $\calA_r(
\hroot,\mu
)$ containing $s_{r-1}$ or $s_{r-2}s_{r-1}$ is also given by $|\calA_{r-2}(\hroot,\mu)|$.
Taking the sum over these cases shows that the desired recurrence still holds.
\end{proof}

\begin{figure}
\begin{center}
    \begin{tikzpicture}[scale=.8]
            \node[text=c1] (A) at (0,-4) {\fbox{$1$}};
            
            \node[text=c1] (B) at (-3,-2) {\fbox{$s_2$}};
            \node[text=c1] (C) at (0,-2) {\fbox{$s_3$}};
            \node[text=c1] (D) at (3,-2) {\fbox{$s_4$}};
            
            \node[text=c1] (E) at (-6,0) {\fbox{$s_3s_2$}};
            \node[text=c1] (F) at (-3,0) {\fbox{$s_2s_3$}};
            \node[text=c1] (G) at (0,0) {\fbox{$s_2s_4$}};
            \node[text=c1] (H) at (3,0) {\fbox{$s_4s_3$}};
            \node[text=c1] (I) at (6,0) {\fbox{$s_3s_4$}};
            
            \node[text=c1] (J) at (-6,2) {\fbox{$s_2s_3s_2$}};
            \node[text=c1] (K) at (0,2) {\fbox{$s_2s_4s_3$}};
            \node[text=c1] (L) at (6,2) {\fbox{$s_3s_4s_3$}};

            \draw[c1, ultra thick] (A) -- (B);
            \draw[c1, ultra thick] (A) -- (C);
            \draw[c1, ultra thick] (A) -- (D);
            
            \draw[c1, ultra thick] (B) -- (E);
            \draw[c1, ultra thick] (B) -- (G);
            \draw[c1, ultra thick] (C) -- (F);
            \draw[c1, ultra thick] (C) -- (H);
            \draw[c1, ultra thick] (D) -- (G);
            \draw[c1, ultra thick] (D) -- (I);
            
            \draw[c1, ultra thick] (E) -- (J);
            \draw[c1, ultra thick] (F) -- (J);
            \draw[c1, ultra thick] (F) -- (K);
            \draw[c1, ultra thick] (H) -- (K);
            \draw[c1, ultra thick] (H) -- (L);
            \draw[c1, ultra thick] (I) -- (L);

            \node[text=c1] (AL) at (-4.5,-2) {$s_5$};
            \node[text=c1] (BL) at (-7.5,0) {$s_2s_5$};
            \node[text=c1] (CL) at (-4.5,0) {$s_3s_5$};
            \node[text=c1] (EL) at (-10.5,2) {$s_3s_2s_5$};
            \node[text=c1] (FL) at (-7.5,2) {$s_2s_3s_5$};
            \node[text=c1] (JL) at (-10.5,4) {$s_2s_3s_2s_5$};

            \draw[c1,  ultra thick] (AL) -- (BL);
            \draw[c1,  ultra thick] (BL) -- (EL);
            \draw[c1,  ultra thick] (CL) -- (FL);
            \draw[c1,  ultra thick] (EL) -- (JL);
            \draw[c1,  ultra thick] (FL) -- (JL);

            \draw[c1,  ultra thick] (A) -- (AL);
            \draw[c1,  ultra thick] (B) -- (BL);
            \draw[c1,  ultra thick] (C) -- (CL);
            \draw[c1,  ultra thick] (E) -- (EL);
            \draw[c1,  ultra thick] (F) -- (FL);
            \draw[c1,  ultra thick] (J) -- (JL);

            \node (AR) at (1.5,-2) {\underline{$s_6$}};
            
            \node (BR) at (-1.5,0) {\underline{$s_2s_6$}};
            \node (CR) at (1.5,0) {\underline{$s_3s_6$}};
            \node (DR) at (4.5,0) {\underline{$s_4s_6$}};
            
            \node (ER) at (-4.5,2) {\underline{$s_3s_2s_6$}};
            \node (FR) at (-1.5,2) {\underline{$s_2s_3s_6$}};
            \node (GR) at (1.5,2) {\underline{$s_2s_4s_6$}};
            \node (HR) at (4.5,2) {\underline{$s_4s_3s_6$}};
            \node (IR) at (7.5,2) {\underline{$s_3s_4s_6$}};
            
            \node (JR) at (-4.5,4) {\underline{$s_2s_3s_2s_6$}};
            \node (KR) at (1.5,4) {\underline{$s_2s_4s_3s_6$}};
            \node (LR) at (7.5,4) {\underline{$s_3s_4s_3s_6$}};

            \foreach \n in {A,B,C,D,E,F,G,H,I,J,K,L}{
                \draw[black, dashed] (\n) -- (\n R);        
    }

            \draw[black] (AR) -- (BR);
            \draw[black] (AR) -- (CR);
            \draw[black] (AR) -- (DR);
            
            \draw[black] (BR) -- (ER);
            \draw[black] (BR) -- (GR);
            \draw[black] (CR) -- (FR);
            \draw[black] (CR) -- (HR);
            \draw[black] (DR) -- (GR);
            \draw[black] (DR) -- (IR);
            
            \draw[black] (ER) -- (JR);
            \draw[black] (FR) -- (JR);
            \draw[black] (FR) -- (KR);
            \draw[black] (HR) -- (KR);
            \draw[black] (HR) -- (LR);
            \draw[black] (IR) -- (LR);
    \end{tikzpicture}
\end{center}
\caption{The Hasse diagram for $\calA_7(\tilde\alpha,\mu)$ in the left weak Bruhat order with $\mu=-\alpha_{2,4}$. The blue subposet with thick edges is $\calA_6(\tilde\alpha,\mu)$, and the subposet of boxed values is $\calA_5(\tilde\alpha,\mu)$. Hence,  $\calA_7(\tilde\alpha,\mu)$ decomposes into two pieces: one being $\calA_6(\tilde\alpha,\mu)$ and the other being a translation of $\calA_5(\tilde\alpha,\mu)$ by $s_6$ along the dotted edges.}\label{fig:recursive_hasse}
\end{figure}
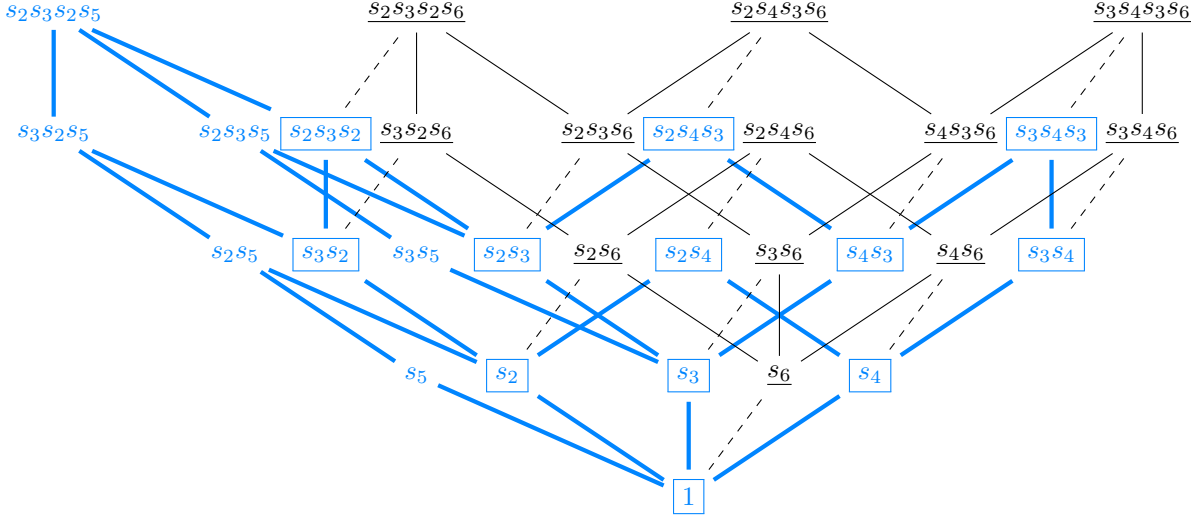

\Cref{fig:recursive_hasse} gives an example of the recursion in \Cref{prop:RecurrenceOnSizesFixedRoot}. To prepare for later computations, we provide a general result concerning the generating function for any sequence defined by the two term homogeneous linear recurrence $a_n=a_{n-1}+a_{n-2}$.

\begin{lemma}\label{lem:FibWIthOtherStartingPOints}
Let $\{a_i\}_{i\geq0}$ be a sequence such that $a_i = 0$ for $0 \leq i < r$, $a_r = c$, $a_{r+1} = d$, and for $i \geq r+2$, $a_i = a_{i-1} + a_{i-2}$. Then,\[
\sum_{i\geq0} a_ix^i = \frac{cx^r + (d-c)x^{r+1}}{1-x-x^2}.
\]
\end{lemma}

\begin{proof}
We follow a standard computation method regarding generating functions for recursively defined sequences. Let $A(x) = \sum_{i\geq0} a_ix^i$. Then, we have  \begin{align*}
A(x) = \sum_{i\geq0} a_ix^i = \sum_{i\geq r} a_ix^i = c x^r + d x^{r+1} + \sum_{i\geq r+2} a_ix^i &=  c x^r + d x^{r+1} + \sum_{i\geq r+2} a_{i-1}x^i + \sum_{i\geq r+2} a_{i-2}x^i\\
&=c x^r + d x^{r+1} + x(A(x) - cx^r) + x^2 A(x).
\end{align*}
Solving for $A(x)$ in the above equation, the result follows.
\end{proof}

Now, we turn our attention to tracking the sizes of Weyl alternation sets simultaneously. To prepare for this result, it is advantageous to consider the set \[P_r^i \coloneqq \{ w \in \mathcal{A}_r(\hroot, -\alpha_{i,r}): r \text{ is not in the influence of }w\}.\] 
In other words, if $r$ is not in the influence of $w\in\calA(\hroot,-\alpha_{i,r})$, then $s_r$ does not appear in a reduced expression for $w$.
We set $p_r^i \coloneqq \vert P_r^i \vert$.  
Our proof for the following result uses the same technique as in \cite[Proposition 4.6]{Lauren}.

\begin{lemma}\label{lem:p_recursion}
Given $r \geq i+4$, 
the sequence $p_r^i$ satisfies the recurrence  
\[
p_r^i = p_{r-1}^i + p_{r-2}^i + 3p_{r-3}^i + p_{r-4}^i.
\]
\end{lemma}

\begin{proof}
To establish the result, we provide a bijection \[\phi:P_{r-1}^i \cup P_{r-2}^i \cup Y_1 \cup Y_2 \cup Y_3 \cup P_{r-4}^i\to P_r^i\] where $Y_1$, $Y_2$, and $Y_3$ are distinct copies of $P_{r-3}^i$. 
First of all, from our description of each Weyl alternation set, 
we can find an isomorphic copy of $P_{r-1}^i$ in the 
set $P_{r}^i$. This inclusion is just the natural inclusion of the symmetric group on $n$ letters into the symmetric group on $n+1$ letters.
Thus, if $w \in P_{r-1}^i$, we set $\phi(w) = w$. 
The cokernel of this map is all $w \in P_r^i$ containing $s_{r-1}$ in their reduced expressions. 
More specifically, the cokernel consists of all $w \in P_r^i$ containing the subword $s_{r-1}$, $s_{r-2}s_{r-1}$, $s_{r-1}s_{r-2}$, $s_{r-2}s_{r-1}s_{r-2}$, or $s_{r-3}s_{r-1}s_{r-2}$. 

Given $w \in P_{r-2}^i$, we set $\phi(w) = w s_{r-1}$. Then $\phi(w) \in P_r^i$ for $w\in P_{r-2}^i$ and the image $\phi(P_{r-2}^i)$ constitutes all elements of $P_r^i$ containing the subword $s_{r-1}$. Next, if $w \in Y_1$, we set $\phi(w) = w s_{r-2}s_{r-1}$. Thus, $\phi(w) \in P_r^i$ for $w\in Y_1$ and the image $\phi(Y_1)$ constitutes all elements of $P_r^i$ containing the subword $s_{r-2}s_{r-1}$. Similarly, if $w \in Y_2$, we set $\phi(w) = w s_{r-1}s_{r-2}$, and if $w \in Y_3$, we set $\phi(w) = w s_{r-2}s_{r-1}s_{r-2}$. {So $\phi(w)\in P_r^i$ for $w\in Y_2$ and $Y_3$, the image $\phi(Y_2)$ constitutes all elements of $P_r^i$ containing the subword $s_{r-1}s_{r-2}$, and the image $\phi(Y_3)$ constitutes all elements of $P_r^i$ containing the subword $s_{r-2}s_{r-1}s_{r-2}$.} Finally, if $w \in P_{r-4}^i$, we set $\phi(w) = w s_{r-3}s_{r-1}s_{r-2}$. {Consequently, $\phi(w)\in P_r^i$ for $w\in P_{r-4}^i$ and the image $\phi(P_{r-4}^i)$ constitutes all elements of $P_r^i$ containing the subword $s_{r-3}s_{r-1}s_{r-2}$.} Thus, $\phi$ is indeed a bijection, which shows that the cardinalities of the sets satisfy the desired recurrence. 
\end{proof}

\begin{lemma}\label{lem:p_diag_recurrence}
If $3<i\leq r$, then 
the $p_r^i$ satisfy the recurrence  
\[
p_r^i = p_{r-1}^{i-1}+p_{r-2}^{i-2}.
\]
\end{lemma}

\begin{proof}
Since $i > 3$, we know $\max(i-1,2) = i-1$ and $\max(i-2,2) = i-2$. Therefore, if we take the description of $\BAS_{r-1}(\hroot,\alpha_{i-1,r-1})\backslash\{s_{r-1}\}$ and add one to each index, we have $\BAS_{r}(\hroot,\alpha_{i-1,r-1})\backslash\{s_2,s_r\}$. 
By the shift in indices, $p_{r-1}^{i-1}$ enumerates the elements of $P_r^i$ 
that do not contain $s_2$ in their reduced expressions.

It remains to enumerate the elements of $P_r^i$ that contain  $s_2$ in their reduced expressions.
Since $i \geq 4$, when we remove the word $s_2$, the resulting Weyl group element is a pairwise independent product of  \begin{enumerate}
\noindent \begin{minipage}{0.45\textwidth}
    \item[(a)] $s_k$ with $3 < k < r$,
    \item[(b)] $s_{k+1}s_k$ with $\max(i-1,4) \leq k \leq r-2$,
    \item[(c)] $s_ks_{k+1}$ with $i \leq k \leq r-2$,
\end{minipage}
\begin{minipage}{0.45\textwidth}
    \item[(d)] $s_ks_{k+1}s_k$ with $i \leq k \leq r-2$, and 
    \item[(e)] $s_{k+2}s_ks_{k+1}$ with $i \leq k \leq r-3$.
\end{minipage}
\end{enumerate}
Note, if we subtract 2 from all indices, we have a description of $\BAS_{r-2}(\hroot,-\alpha_{i-2,r-2})\backslash\{s_{r-2}\}$. This means that the number of elements of $P_r^i$  that contain  $s_2$ in their reduced expressions is precisely $p_{r-2}^{i-2}$, completing the proof.
\end{proof}

\begin{lemma}
    If $1\leq i< r$, then $\vert\calA_r(\hroot,-\alpha_{i,r-1})\vert=p_r^i$.
\end{lemma}

\begin{proof}
    Let $i=1$. Then, comparing the inequalities in \Cref{thm:BAS_for_negAlpha} to those in \Cref{cor:BAS1j} with $j=r-1$, we have that \[\BAS_r(\hroot,-\alpha_{1,r-1})=\BAS_r(\hroot,{-}\alpha_{1,r})\setminus\{s_r\}.\] Similarly, for $1<i<r$, by comparing the inequalities in \Cref{cor:BASir} to those in \Cref{cor:BASij} with $j=r-1$, we have also that \[\BAS_r(\hroot,-\alpha_{i,r-1})=\BAS_r(\hroot,{-}\alpha_{i,r})\setminus\{s_r\}.\]
    The result then follows from $p_r^i$ being the number of  elements in $\calA(\hroot,-\alpha_{i,r})$ without $s_r$ in their reduced expressions and the fact that the only basic allowable subword containing $s_r$ is the element $s_r$ itself.
\end{proof}

To prepare for our next results, we set $h_r^i = \vert \mathcal{A}_r(\hroot, -\alpha_{i,r})\vert$. Next, we show that the quantities $h_r^i$ and $p_r^i$ are closely related.

\begin{lemma}\label{lem:p_is_i_to_r-1}
If $r\geq 2$ and $1\leq i \leq r-1$, then $h_r^i = p_r^i + p_{r-1}^i$.
\end{lemma}

\begin{proof}
To establish the result, we provide a bijection between $P_r^i \cup P_{r-1}^i$ and $\mathcal{A}_r(\hroot, -\alpha_{i,r})$. Given $w \in P_r^i$, we set $\phi(w) = w$. By definition, the image $\phi(P_r^i)$ consists of all $w \in \mathcal{A}_r(\hroot, -\alpha_{i,r})$ without a subword $s_r$, of which there are $p_r^i$. Given $w \in P_{r-1}^i$, we set $\phi(w) = w s_r$. One can check that $w s_r \in \mathcal{A}_r(\hroot, -\alpha_{i,r})$, and moreover that all $\tau \in \mathcal{A}_r(\hroot, -\alpha_{i,r})$ with a subword $s_r$ can be written this way. Recall that there are $p_{r-1}^i$ such elements of $\mathcal{A}_r(\hroot, -\alpha_{i,r})$. As these subsets of $\mathcal{A}_r(\hroot, -\alpha_{i,r})$ are disjoint, $\phi$ defines the desired bijection, and we get $h_r^i = p_r^i + p_{r-1}^i$.
\end{proof}

Now that we know recurrences for the sequence $\{p_r^i\}_r$ and a relationship between the sequences $\{p_r^i\}_r$ and $\{h_r^i\}_r$, we can conclude that the two sequences satisfy the same recurrences.

\begin{corollary}\label{cor:h_recursion}
If $r \geq i+4$, then \[h_r^i = h_{r-1}^i + h_{r-2}^i + 3h_{r-3}^i + h_{r-4}^i.
\] If $i,r\geq 3$, 
then \[h^i_r=h^{i-1}_{r-1}+h^{i-2}_{r-2}.\]
\end{corollary}

By computing initial conditions using \Cref{thm:characterize_BAS} and \Cref{cor:BASir}, we provide generating functions for both of these sequences.

\begin{lemma}\label{lem:GenFxnPiAndHi}
Define $\calP^i(x) = \displaystyle\sum_{r \geq i} p_r^i x^r$ and $\calH^i(x) = \displaystyle\sum_{r \geq i} h_r^i x^r$.

\begin{enumerate}
    \item If $i = 1$, then \[
\calP^1(x) = \frac{x + x^2}{1 - x - x^2 - 3x^3 - x^4} \qquad
\mbox{and} \qquad 
\calH^1(x) = \frac{x + 2x^2 + x^3}{1 - x - x^2 - 3x^3 - x^4}
.\]
\item If $i = 2$, then 
\[
\calP^2(x) = \frac{x^2 + x^{3} + 3x^{4} + x^{5}}{1 - x - x^2 - 3x^3 - x^4} 
\qquad 
\mbox{and}
\qquad 
\calH^2(x) = \frac{2x^3 + x^{4} + 3x^{5} + x^{6}}{1 - x - x^2 - 3x^3 - x^4}
.\]
\item If $i=3$, then \[
\calP^3(x) = \frac{2x^3 + 2x^{4} + 3x^{5} + x^{6}}{1-x-x^2 - 3x^3 - x^4} 
\qquad 
\mbox{and}
\qquad \calH^3(x) = \frac{3x^3 + 3x^{4} + 4 x^{5} + x^{6}}{1-x-x^2 - 3x^3 - x^4}
.\]
\end{enumerate}

Furthermore, defining $\calP(x,s)=\sum_{i\geq1} \calP^i(x)s^i$ and $\calH(x,s)=\sum_{i\geq1}\calH^i(x)s^i$, we have \begin{align*}
\calP(x,s)&=\frac{xs(x^4s+3x^3s+x+1)}{(1-x-x^2-3x^3-x^4)(1-xs-(xs)^2)}\\
\intertext{and}
\calH(x,s)&=\frac{xs(x^5s+3x^4s-xs+x^2+2x+1)}{(1-x-x^2-3x^3-x^4)(1-xs-(xs)^2)}.\\
\end{align*}
\end{lemma}

\begin{proof}
For the cases $i=1,~2$, and $3$, reasoning as in the proof of \Cref{lem:FibWIthOtherStartingPOints},  the recursions in \Cref{lem:p_recursion} and \Cref{cor:h_recursion} guarantee these generating functions must have the form \[\frac{ax^i+bx^{i+1}+cx^{i+2}+dx^{i+3}}{1-x-x^2-3x^3-x^4}=ax^i+(a+b)x^{i+1}+(2a+b+c)x^{i+2}+(6a+2b+c+d)x^{i+3}+\cdots\] with coefficients $a,b,c$, and $d$ depending on the initial conditions. 
The formulas for these cases are then obtained by first computing the values of $p^i_r$ and $h^i_r$ for the first four values as shown in \Cref{tab:initial_values_p_h}, then solving for the coefficients $a,b,c$, and $d$. 

\begin{table}[]
    \centering
    \begin{tabular}{c||c|c|c|c||c|c|c|c}
        $i$ & $p^i_{i}$ & $p^i_{i+1}$ & $p^i_{i+2}$ & $p^i_{i+3}$ & $h^i_{i}$ & $h^i_{i+1}$ & $h^i_{i+2}$ & $h^i_{i+3}$ \\ \hline
         1 & 1 & 2 & 3 & 8 & 1 & 3 & 5 & 11\\
         2 & 1 & 2 & 6 & 12 & 2 & 3 & 8 & 18\\
         3 & 2 & 4 & 9 & 20 & 3 & 6 & 13 & 29
    \end{tabular}
    \caption{The values of $p^i_r$ and $h^i_r$ for $i\in\{1,2,3\}$ and $i\leq r\leq i+2$.}
    \label{tab:initial_values_p_h}
\end{table}

We now apply the recursion $p^i_r=p^{i-1}_{r-1}+p^{i-2}_{r-2}$ from \Cref{lem:p_diag_recurrence} for $r\geq i\geq 3$ to obtain a formula for $\calP(x,s)$. 
For $i\geq 3$, we first compute\begin{align*}
    \calP^i(x)&=\sum_{r\geq i}p_r^i x^r\\
    &=\sum_{r\geq i} (p_{r-1}^{i-1}+p_{r-2}^{i-2})x^r\\
    &=\sum_{r\geq i}p_{r-1}^{i-1}x^r+\sum_{r\geq i}p_{r-2}^{i-2}x^r\\
    &=x\sum_{r\geq i-1}p_r^{i-1}x^r+x^2\sum_{r\geq i-2}p_r^{i-2}x^r\\
    &=x\calP^{i-1}(x)+x^2\calP^{i-2}(x).
\end{align*}

With this recursion in hand, we compute \begin{align*}
    \calP(x,s)&=\sum_{i\geq 1}\calP^i(x)s^i\\
    &=\calP^1(x)s+\calP^2(x)s^2+\sum_{i\geq 3} \calP^i(x)s^i\\
    &=\calP^1(x)s+\calP^2(x)s^2+\sum_{i\geq3}x\calP^{i-1}(x)s^i+\sum_{i\geq3}x^2\calP^{i-2}(x)s^i\\
    &=\calP^1(x)s+\calP^2(x)s^2+xs(\calP(x,s)-\calP^1(x)s)+(xs)^2\calP(x,s).
\end{align*}

Now, note $ (1-xs-(xs)^2)\calP(x,s)=(1-xs)\calP^1(x)s+\calP^2(x)s^2$, and 
\begin{align*}
    \calP(x,s)&=\frac{(1-xs)\calP^1(x)s+\calP^2(x)s^2}{1-xs-(xs)^2}\\
    &=\frac{(1-xs)(x+x^2)s+(x^2+x^3+3x^4+x^5)s^2}{(1-x-x^2-3x^3-x^4)(1-xs-(xs)^2)}\\
    &=\frac{xs(x^4s+3x^3s+x+1)}{(1-x-x^2-3x^3-x^4)(1-xs-(xs)^2)}.
\end{align*}

We handle $\calH(x,s)$ similarly, noting that by \Cref{cor:h_recursion}, we also have that \[\calH^i(x)=x\calH^{i-1}(x)+x^2\calH^{i-2}(x)\] for $i\geq3$. We follow the above computation until \begin{align*}
    \calH(x,s)&=\frac{(1-xs)\calH^1(x)s+\calH^2(x)s^2}{(1-xs-(xs)^2)}\intertext{where we insert the initial conditions for $\calH^i(x)$ to obtain}
    \calH(x,s)&=\frac{(1-xs)(x+2x^2+x^3)s+(2x^3+x^4+3x^5+x^6)s^2}{(1-x-x^2-3x^3-x^4)(1-xs-(xs)^2)}\\
    &=\frac{xs(x^5s+3x^4s-xs+x^2+2x+1)}{(1-x-x^2-3x^3-x^4)(1-xs-(xs)^2)}.\qedhere
\end{align*}
\end{proof}

We remark here that the sequence $\{h_r^1\}$ coincides with \cite[\seqnum{A196423}]{OEIS}. This sequence also counts the cardinality of the following family of sets. Let $X_r$ denote the set of length $r$ sequences $x_1,x_2,\ldots,x_r$ where $x_i \in \{0,1,2\}$ and $x_i = \vert \{x_j : \vert j - i \vert = 1 \text{ and } x_j < x_i\}\vert$. 
For example, $X_3 = \{000, 001, 100, 101,020\}$. We show there is a relationship between $X_r$ and $\mathcal{A}_r(\hroot, -\hroot)$.

\begin{theorem} 
There is a bijection between  $X_r$ and $\mathcal{A}_r(\hroot,-\hroot)$. 
\end{theorem}

\begin{proof}
Observe that, an element of $X_r$ can be decomposed into the following (consecutive) subsequences whose neighbors are all 0's  \[
\tau_1:= 1, \quad \tau_2:= 2, \quad
\tau_3:= 11, \quad \tau_4:= 
12,\quad \tau_5:= 21,\quad \mbox{and}\quad\tau_6:= 121.
\]
Notice that $\tau_1$ can only appear at the beginning or end of an element of $X_r$, and no other subsequence $\tau_i$ can appear at the beginning or end of the sequence in $X_r$. Given $x \in X_r$, define $\psi(x) \in A_r$ in the following way. We begin by setting $\psi(x) = id$. Now, we check which subsequences $\tau_i$ are contained in $x$ and at what indices they appear.
\begin{enumerate}
    \item If $x$ has a subsequence $\tau_1$ in position $i$, then  we multiply $\psi(x)$ by $s_i$. Necessarily, $i = 1$ or $i = r$ as discussed.
    \item If $x$ has a subsequence $\tau_2$ in position $i$, then  we multiply $\psi(x)$ by $s_i$.  
    \item If $x$ has a subsequence $\tau_3$ (resp. $\tau_4, \tau_5$) in positions $i,i+1$, we multiply $\psi(x)$ by $s_is_{i+1}$(resp $s_{i+1}s_i$, $s_is_{i+1},s_i$).
    \item If $x$ has a subsequence $\tau_6$ in positions $i,i+1,i+2$, we multiply $\psi(x)$ by $s_is_{i+2}s_{i+1}$.
\end{enumerate}

Comparing our method of decomposing an element of $X_r$ with \Cref{thm:BAS_for_negAlpha} shows that $\psi(x) \in \mathcal{A}_r(\hroot,-\hroot)$. Moreover, an inverse to $\psi$ can be constructed in a similar manner, verifying that this map is a bijection. 
\end{proof}

We conclude this section by  providing a generating function that records the cardinalities of all Weyl alternation sets in type $A$ for negative roots. 

\begin{theorem}\label{thm:gen_func} If $\hroot$ is the highest root and $\mu=-\alpha_{i,j}$ with $1\leq i\leq j\leq r$
 is a negative root of the Lie algebra of type $A_r$, then we have the generating function
\begin{align*}
\sum_{1\leq i\leq j\leq r} \vert \mathcal{A}_r(\hroot, -\alpha_{i,j})\vert x^r s^it^j = \frac{1}{t(1-x-x^2)}\left((1-x)t\calH(xt,s)+\calP(xt,s)-\frac{xst}{1-xst-(xst)^2}\right),
\end{align*} where $\calH(xt,s)$ and $\calP(xt,s)$ are as in \Cref{lem:GenFxnPiAndHi}.
\end{theorem}

\begin{proof}
For a fixed pair $1 \leq i\leq j$, from \Cref{prop:RecurrenceOnSizesFixedRoot} and \Cref{lem:FibWIthOtherStartingPOints}, we have 
\begin{align*}
\sum_{r \geq j} \vert \mathcal{A}_r(\hroot, -\alpha_{i,j})\vert x^r
&= \frac{h_j^i x^j + (\vert \mathcal{A}_{j+1}(\hroot, -\alpha_{i,j})\vert - h_j^i)x^{j+1}}{1-x-x^2}\\
&= \frac{h_j^i x^j + (p_{j+1}^i - h_j^i)x^{j+1}}{1-x-x^2},
\end{align*}
where, by \Cref{lem:p_is_i_to_r-1}, we can identify $\mathcal{A}_{j+1}(\hroot, -\alpha_{i,j})$ with $P_{j+1}^i$. 

Now, if we fix $i \geq 1$ and sum over all $j$, we have 
\begin{align*}
\sum_{j \geq i} \bigg(\frac{h_j^i x^j + (p_{j+1}^i - h_j^i)x^{j+1}}{1-x-x^2}\bigg)t^j
&=\frac{1}{1-x-x^2} \bigg( \sum_{j \geq i} h_j^i x^j t^j + \sum_{j \geq i} p_{j+1}^i x^{j+1}t^j - \sum_{j \geq i} h_j^i x^{j+1}t^j\bigg) \\
&= \frac{1}{1-x-x^2}\bigg((1-x) \calH^i(xt) + \frac{\calP^i(xt) - p_i^i (xt)^i}{t}\bigg).
\end{align*}

Finally, we must take the resulting expression and sum over all $i \geq 1$, yielding \begin{align*}
\sum_{1\leq i\leq j\leq r}\vert \mathcal{A}_r(\hroot, -\alpha_{i,j})\vert x^r s^it^j&=\sum_{i \geq 1}\frac{1}{1-x-x^2}\bigg((1-x) \calH^i(xt) + \frac{\calP^i(xt) - p_i^i (xt)^i}{t}\bigg)s^i \\
&=\frac{1}{t(1-x-x^2)}\bigg((1-x)t \sum_{i \geq 1}\calH^i(xt)s^i + \sum_{i \geq 1}\calP^i(xt)s^i -\sum_{i \geq 1} p_i^i(xst)^i\bigg).
\end{align*}

We understand the first two summations from \Cref{lem:GenFxnPiAndHi}. We can compute $\sum_{i \geq 1}p_i^i(xt)^is^i$ by combining \Cref{lem:FibWIthOtherStartingPOints} and \Cref{lem:p_diag_recurrence}.
From \Cref{lem:GenFxnPiAndHi}, we have the starting values $p_1^1 = p_2^2 = 1$. Therefore, we can conclude

\begin{align*}
\sum_{1\leq i\leq j\leq r}\vert \mathcal{A}_r(\hroot, -\alpha_{i,j})\vert x^r s^it^j&=\frac{1}{t(1-x-x^2)}\left((1-x)t\calH(xt,s)+\calP(xt,s)-\frac{p_1^1xst+(p_2^2-p_1^1)(xst)^2}{1-xst-(xst)^2}\right)\\
&= \frac{1}{t(1-x-x^2)}\left((1-x)t\calH(xt,s)+\calP(xt,s)-\frac{xst}{1-xst-(xst)^2}\right).\qedhere
\end{align*}

\end{proof}

\section{Future work}\label{sec:future}
Lusztig introduced the $q$-analog of Kostant's partition function as the polynomial-valued function 
    $\wp_q(\xi)=c_0+c_1q+c_2q^2+\cdots + c_kq^k$,
where $c_i $ equals the number of ways to write $\xi$ as a sum of exactly $i$ positive roots \cite[Proposition 9.2]{Lusztig}.
Then the $q$-analog of Kostant's weight multiplicity formula is given by
     \[m_q(\lambda, \mu)=\sum_{\sigma \in W} (-1)^{\ell(\sigma)}\wp_q(\sigma(\lambda+\rho)-\mu - \rho).\]
In type $A_r$, using the Weyl alternation set $\calA(\hroot,0)$, Harris gives a completely combinatorial proof that $m_q(\hroot,0)=\sum_{i=1}^r q^i$ \cite[Proposition~3.3]{PHThesisPublication}.
Again in type $A_r$, using the Weyl alternation set $\calA(\hroot,\mu)$ with $\mu\in\Phi^+$, Harry shows that $m_q(\hroot,\mu)$ is a power of $q$ \cite[Theorem~4.1]{Harry}.
Harry conjectured the following.
\begin{conjecture}\label{big conjecture}
     If $\hroot$ is the highest root and $\mu=-\alpha_{i,j}$ with $1\leq i\leq j\leq r$
 is a negative root of the Lie algebra of type $A_r$, then
\[m_q(\hroot,\mu)=q^{r+j-i+1}+q^{r+j-i}-q^{j-i+1}.\]
\end{conjecture}
At this moment, we have a proof of \Cref{big conjecture} for $\mu=-\alpha_i$ with $1\leq i\leq r$. 
Given our characterization of the Weyl alternation sets $\calA(\hroot,\mu)$ for negative roots $\mu\in\Phi^-$, in future work, we aim to use similar techniques to give a full proof of \Cref{big conjecture}.

\section*{Acknowledgments}

 This material is based on work supported by the National Science Foundation under Grant Number DMS 1916439 while the authors were participating in the Mathematics Research Communities (MRC) 2024 Summer Conference at Beaver Hollow Conference Center in Java Center, New York. We extend our thanks to the American Mathematical Society for their support. We also thank the reviewers who provided helpful feedback on an abbreviated version of the article which was submitted to FPSAC 2025.
 
 E. Banaian was partially supported by Research Project 2 from the Independent Research Fund
Denmark (grant no. 1026-00050B).
This work was supported by a grant from the Simons Foundation (Travel Support for Mathematicians, P.\ E.\ Harris).

\appendix 

\begin{appendices}
\section{Proofs} \label{app:proofs}

We recall Lemma \ref{lem:ProductOfTwoBASs}.

\begin{lemma}
    Let $S$ be the set of elements listed in \Cref{prop:mu=-hroot}.
    Let $\sigma, \tau\in S$ be nonindependent elements.
    Then, the product $\sigma\tau$ falls into one of the following three cases:
    \begin{enumerate}
        \item $\sigma\tau\in S$,
        \item $\sigma\tau=\nu_1\nu_2\cdots\nu_m$ where $\{\nu_1,\nu_2,\ldots,\nu_m\}$ is a (possibly empty) pairwise independent subset of $S$ and \\
        $\ell(\nu_1)+\cdots+\ell(\nu_m)<\ell(\sigma)+\ell(\tau)$, or
        \item $\sigma\tau$ contains a forbidden subword listed in \Cref{lem:forbidden words} or~\Cref{Highest:1234}.
    \end{enumerate}
\end{lemma}

\begin{proof}
    We proceed via cases. Let $\sigma$ and $\tau$ come from one of the following (not necessarily distinct) collections:
\begin{enumerate}
        \item[(a)] $s_k$ with $1 \le k \le r$,
        \item[(b)] $s_{k+1}s_{k}$ with $2 \le k \le r-2$,
        \item[(c)] $s_{k}s_{k+1}$ with $2 \le k \le r-2$, 
        \item[(d)] $s_{k}s_{k+1}s_{k}$ with $2 \le k \le r-2$, or
        \item[(e)] $s_{k+2}s_{k}s_{k+1}$ with $2 \le k \le r-3$.
    \end{enumerate}
    For each pair of forms that $\sigma, \tau\in S$ can respectively take, we fix the indices of $\sigma$, then consider the range of indices of $\tau$ for which $\sigma$ and $\tau$ are not independent.

    To reduce the number of cases that need to be handled, consider the automorphism of the Weyl group $f:W\to W$ given by the conjugation $f(\sigma)=w_0\sigma w_0$ by the longest element $w_0$ of $W$. This map $f$ sends $s_i$ to $s_{r+1-i}$ and preserves length.  
    Note additionally that $f(S)=S$ and words of the form (a), (d), and (e) are sent to words of the same form. Words of the form (b) are taken to words of the form (c) by $f$ and vice versa. If $\sigma$ is a forbidden word as in \Cref{lem:forbidden words}, then so is $f(\sigma)$. 
    As the automorphism $f$ preserves words of the form (a), (d), and (e), while replacing a word of the form (c) with a word of the form (b), we can reduce the number of cases we consider.
\smallskip

\noindent \textbf{Products of the form (a)(a):}
    Let $\sigma = s_k$ for some $1 \leq k \leq r$. For $\tau = s_{j}$ with $1\leq j\leq r$ not to be independent with $\sigma$, we must have $j\in\{k-1,k,k+1\}$.
    \begin{itemize}
        \item Let $j=k-1$ for $2\leq k\leq r$. Then $\sigma\tau=s_ks_{k-1}$. If $k=2$ (resp., $k=r$), then $s_2s_1$ (resp., $s_rs_{r-1}$) is forbidden by \Cref{lem:forbidden words}. That is, $\sigma\tau$ falls into case (3). Otherwise, $3\leq k\leq r-1$ and $s_ks_{k-1}\in S$. That is, $\sigma\tau$ falls into case (1). 
        \item Let $j=k$ for $1\leq k\leq r$. Then $\sigma\tau=s_ks_k=1$, a product of an empty subset of $S$, with $0=\ell(1)=\ell(s_ks_{k})<\ell(s_k)+\ell(s_{k})=2$. That is, $\sigma\tau$ falls into case (2).
        \item Let $j=k+1$ for $1\leq k\leq r-1$. Then $\sigma\tau=s_ks_{k+1}$. If $k=1$ (resp., $k=r-1$), then $s_1s_2$ (resp., $s_{r-1}s_r$) is forbidden by \Cref{lem:forbidden words}. That is, $\sigma\tau$ falls into case (3). Otherwise, $2\leq k\leq r-2$ and $s_ks_{k+1}\in S$. That is, $\sigma\tau$ falls into case (1).
    \end{itemize}
\smallskip

\noindent \textbf{Products of the form (b)(b) (equivalent to (c)(c) via $f$):}
    Let $\sigma= s_{k+1}s_k$ for some $2\leq k\leq r-2$. For $\tau=s_{j+1}s_j$ with $2\leq j\leq r-2$ not to be independent with $\sigma$, we must have $j\in\{k-2,k-1,k,k+1,k+2\}$.
    \begin{itemize}
        \item Let $j=k-2$ for $4\leq k\leq r-2$. Then $\sigma\tau=(s_{k+1}s_k)(s_{k-1}s_{k-2})$ contains $s_ks_{k-1}s_{k-2}$, forbidden by \Cref{lem:forbidden words}. That is, $\sigma\tau$ falls into case (3).
        \item Let $j=k-1$ for $3\leq k\leq r-2$. Then $\sigma\tau=(s_{k+1}s_k)(s_ks_{k-1})=s_{k+1}s_{k-1}$ is a pairwise independent product with $2=\ell(s_{k+1})+\ell(s_{k-1})<\ell(\sigma)+\ell(\tau)=4$. That is, $\sigma\tau$ falls into case (2).
        \item Let $j=k$ for $2\leq k\leq r-2$. Then $\sigma\tau=(s_{k+1}s_k)(s_{k+1}s_k)=s_{k+1}s_ks_{k+1}s_k=s_{k+1}s_{k+1}s_ks_{k+1}=s_ks_{k+1}$, a basic allowable subword of length 2. That is, $\sigma\tau$ falls into case (2).
        \item Let $j=k+1$ for $2\leq k\leq r-3$. Then $\sigma\tau=(s_{k+1}s_k)(s_{k+2}s_{k+1})$ contains $s_{k+1}s_k s_{k+2}$ forbidden by \Cref{lem:forbidden words}. That is, $\sigma\tau$ falls into case (3).
        \item Let $j=k+2$ for $2\leq k\leq r-4$. Then $\sigma\tau=(s_{k+1}s_k)(s_{k+3}s_{k+2})=s_{k+3}s_{k+1}s_ks_{k+2}$ contains $s_{k+1}s_k s_{k+2}$ forbidden by \Cref{lem:forbidden words}. That is, $\sigma\tau$ falls into case (3).
    \end{itemize}
\smallskip

\noindent \textbf{Products of the form (d)(d):}
    Let $\sigma = s_ks_{k+1}s_k$ for some $2 \le k \le r-2$. For $\tau= s_js_{j+1}s_j$ with $2 \le j \le r-2$ not to be independent with $\sigma$, we must have $j\in\{k-2,k-1,k,k+1\}$.
    \begin{itemize}
        \item Let $j=k-2$ for $4\leq k\leq r-2$. Then $\sigma\tau=(s_ks_{k+1}s_k)(s_{k-2}s_{k-1}s_{k-2})$ contains $s_{k+1}s_ks_{k-2}s_{k-1}$, forbidden by \Cref{Highest:1234}. That is, $\sigma\tau$ falls into case (3).
        \item Let $j=k-1$ for $3\leq k\leq r-2$. Then $\sigma\tau=(s_ks_{k+1}s_k)(s_{k-1}s_ks_{k-1})$ contains $s_{k+1}s_ks_{k-1}$, forbidden by \Cref{lem:forbidden words}. That is, $\sigma\tau$ falls into case (3).
        \item Let $j=k$ for $2\leq k\leq r-2$. Then $\sigma\tau=(s_ks_{k+1}s_k)(s_ks_{k+1}s_k)=1$, a product of an empty subset of $S$, with $0=\ell(1)=\ell(s_ks_{k})<\ell(s_k)+\ell(s_{k})=2$. That is, $\sigma\tau$ falls into case (2).
        \item Let $j=k+1$ for $2\leq k\leq r-3$. Then  $\sigma\tau=(s_ks_{k+1}s_k)(s_{k+1}s_{k+2}s_{k+1})$ contains $s_ks_{k+1}s_{k+2}$, forbidden by \Cref{lem:forbidden words}. That is, $\sigma\tau$ falls into case (3).
    \end{itemize}
    \smallskip
    
\noindent \textbf{Products of the form (e)(e):}
    Let $\sigma = s_{k+2}s_ks_{k+1}$ for some $2 \le k \le r-3$. For $\tau=s_{j+2}s_js_{j+1}$ with $2 \le j \le r-3$ not to be independent with $\sigma$, we must have $j\in\{k-3,k-2,k-1,k,k+1,k+2,k+3\}$.
    \begin{itemize}
        \item Let $j=k-3$ for $5\leq k\leq r-3$. Then $\sigma\tau=(s_{k+2}s_ks_{k+1})(s_{k-1}s_{k-3}s_{k-2})$ contains $s_{k+2}s_ks_{k+1}s_{k-1}$, forbidden by \Cref{Highest:1234}. That is, $\sigma\tau$ falls into case (3).
        \item Let $j=k-2$ for $4\leq k\leq r-3$. Then $\sigma\tau=(s_{k+2}s_ks_{k+1})(s_{k}s_{k-2}s_{k-1})$ contains $s_{k+1}s_ks_{k-2}s_{k-1}$, forbidden by \Cref{Highest:1234}. That is, $\sigma\tau$ falls into case (3).
        \item Let $j=k-1$ for $3\leq k\leq r-3$. Then $\sigma\tau=(s_{k+2}s_ks_{k+1})(s_{k+1}s_{k-1}s_{k})=(s_{k+2})(s_ks_{k-1}s_k)$ is a pairwise independent product with $4=\ell(s_{k+2})+\ell(s_{k}s_{k-1}s_k)<\ell(\sigma)+\ell(\tau)=6$. That is, $\sigma\tau$ falls into case (2).
        \item Let $j=k$ for $2\leq k\leq r-3$. Then $\sigma\tau=(s_{k+2}s_ks_{k+1})(s_{k+2}s_ks_{k+1})$ contains $s_ks_{k+1}s_{k+2}$, forbidden by \Cref{lem:forbidden words}. That is, $\sigma\tau$ falls into case (3).
        \item Let $j=k+1$ for $2\leq k\leq r-4$. Then $\sigma\tau=(s_{k+2}s_ks_{k+1})(s_{k+3}s_{k+1}s_{k+2})$ contains $s_{k+2}s_ks_{k+1}s_{k+3}$, forbidden by \Cref{Highest:1234}. That is, $\sigma\tau$ falls into case (3).
        \item Let $j=k+2$ for $2\leq k\leq r-5$. Then $\sigma\tau=(s_{k+2}s_ks_{k+1})(s_{k+4}s_{k+2}s_{k+3})$ contains $s_{k+1}s_{k+4}s_{k+2}s_{k+3}$, forbidden by \Cref{Highest:1234}. That is, $\sigma\tau$ falls into case (3).
        \item Let $j=k+3$ for $2\leq k\leq r-6$. Then $\sigma\tau=(s_{k+2}s_ks_{k+1})(s_{k+5}s_{k+3}s_{k+4})=(s_{k+2}s_ks_{k+1})(s_{k+3}s_{k+5}s_{k+4})$ contains $s_{k+2}s_ks_{k+1}s_{k+3}$, forbidden by \Cref{Highest:1234}. That is, $\sigma\tau$ falls into case (3). 
    \end{itemize}
    
\smallskip

\noindent \textbf{Products of the form (a)(b) or (b)(a) (equivalent to (a)(c) or (c)(a) via $f$):}
    Let $\sigma = s_k$ for some $1 \leq k \leq r$. For $\tau = s_{j+1}s_{j}$ with $2\leq j\leq r-2$ not to be independent with $\sigma$, we must have $j\in\{k-2,k-1,k,k+1\}$.
    \begin{itemize}
        \item Let $j=k-2$ for $4\leq k\leq r$. Then \begin{itemize}
            \item $\sigma\tau=(s_k)(s_{k-1}s_{k-2})=s_ks_{k-1}s_{k-2}$ is forbidden by \Cref{lem:forbidden words}, and
            \item $\tau\sigma=(s_{k-1}s_{k-2})s_k=s_{k-1}s_{k-2}s_k$ is forbidden by \Cref{lem:forbidden words}.
        \end{itemize}
        {That is, both $\sigma\tau$ and $\tau\sigma$ fall into case (3).}
        \item Let $j=k-1$ for $3\leq k\leq r-1$. Then \begin{itemize}
            \item $\sigma\tau=(s_k)(s_k s_{k-1})=s_{k-1}\in S$, and
            \item $\tau\sigma=s_ks_{k-1}s_k\in S$.
        \end{itemize}
        {That is, both $\sigma\tau$ and $\tau\sigma$ fall into case (1).}
        \item Let $j=k$ for $2\leq k\leq r-2$. Then \begin{itemize}
            \item $\sigma\tau=(s_k)(s_{k+1}s_k)=s_k s_{k+1} s_k\in S$, and
            \item $\tau\sigma=(s_{k+1}s_k)s_k=s_{k+1}\in S$.
        \end{itemize} 
        {That is, both $\sigma\tau$ and $\tau\sigma$ fall into case (1).}
        \item Let $j=k+1$ for $1\leq k\leq r-3$. Then\begin{itemize}
            \item $\sigma\tau=(s_k)(s_{k+2}s_{k+1})=s_k s_{k+2}s_{k+1}=s_{k+2}s_k s_{k+1}\in S$ if $1<k\leq r-3$ and contains the forbidden $s_1s_2$ when $k=1$, and
            \item $\tau\sigma=(s_{k+2}s_{k+1})s_k=s_{k+2}s_{k+1}s_k$ is forbidden by \Cref{lem:forbidden words}.
        \end{itemize}
        That is, both $\sigma\tau$ and $\tau\sigma$ fall into case (3) when $k=1$, while $\sigma\tau$ falls into case (1) and $\tau\sigma$ falls into case (3) when $1<k\le r-3$.
    \end{itemize}

\smallskip

\noindent \textbf{Products of the form (a)(d) or (d)(a):}
    Let $\sigma=s_k$ for some $1\leq k\leq r$. For $\tau=s_js_{j+1}s_j$ with $2\leq j\leq r-2$ not to be independent with $\sigma$, we must have $j\in\{k-2, k-1, k, k+1\}$.
    \begin{itemize}
        \item Let $j=k-2$ for $4\leq k\leq r$. Then \begin{itemize}
            \item $\sigma\tau=s_k(s_{k-2}s_{k-1}s_{k-2})$ contains $s_ks_{k-1}s_{k-2}$, forbidden by \Cref{lem:forbidden words}, and
            \item $\tau\sigma=(s_{k-2}s_{k-1}s_{k-2})s_k=(s_{k-1}s_{k-2}s_{k-1})s_k$
            contains $s_{k-2}s_{k-1}s_k$, forbidden by \Cref{lem:forbidden words}.
        \end{itemize}
        That is, both $\sigma\tau$ and $\tau\sigma$ fall into case (3).
        \item Let $j=k-1$ for $3\leq k\leq r-1$. Then \begin{itemize}
            \item $\sigma\tau=s_k(s_{k-1}s_ks_{k-1})=s_k(s_ks_{k-1}s_k)=s_{k-1}s_k\in S$, and
            \item $\tau\sigma=(s_{k-1}s_ks_{k-1})s_k=(s_ks_{k-1}s_k)s_k=s_ks_{k-1}\in S$.
        \end{itemize}
        That is, both $\sigma\tau$ and $\tau\sigma$ fall into case (1).
        \item Let $j=k$ for $2\leq k\leq r-2$. Then \begin{itemize}
            \item $\sigma\tau=s_k(s_ks_{k+1}s_k)=s_{k+1}s_k\in S$, and
            \item $\tau\sigma=(s_ks_{k+1}s_k)s_k=s_ks_{k+1}\in S$.
        \end{itemize}
        That is, both $\sigma\tau$ and $\tau\sigma$ fall into case (1).
        \item Let $j=k+1$ for $1\leq k\leq r-3$. Then \begin{itemize}
            \item $\sigma\tau=s_k(s_{k+1}s_{k+2}s_{k+1})$ contains $s_ks_{k+1}s_{k+2}$, forbidden by \Cref{lem:forbidden words}, and
            \item $\tau\sigma=(s_{k+1}s_{k+2}s_{k+1})s_k$ contains $s_{k+2}s_{k+1}s_k$, forbidden by \Cref{lem:forbidden words}.
        \end{itemize}
        That is, both $\sigma\tau$ and $\tau\sigma$ fall into case (3).
    \end{itemize}
\smallskip
\noindent \textbf{Products of the form (a)(e) or (e)(a):}
    Let $\sigma=s_k$ for some $1\leq k\leq r$. For $\tau=s_{j+2}s_{j}s_{j+1}$ with $2\leq j\leq r-3$ not to be independent with $\sigma$, we must have $j\in\{k-3, k-2, k-1, k, k+1\}$.
    \begin{itemize}
        \item Let $j=k-3$ for $5\leq k\leq r$. Then \begin{itemize}
            \item $\sigma\tau= (s_{k})(s_{k-1}s_{k-3}s_{k-2})$ is forbidden by \Cref{Highest:1234}, and
            \item $\tau\sigma=(s_{k-1}s_{k-3}s_{k-2})(s_{k})$ is forbidden by \Cref{Highest:1234}.
        \end{itemize}
        That is, both $\sigma\tau$ and $\tau\sigma$ fall into case (3).
        \item Let $j=k-2$ for $4\leq k\leq r-1$. Then \begin{itemize}
            \item $\sigma\tau= (s_{k})(s_{k}s_{k-2}s_{k-1}) = s_{k-2}s_{k-1}\in S$, and
            \item $\tau\sigma= (s_{k}s_{k-2}s_{k-1})(s_{k})$ contains $s_{k-2}s_{k-1}s_k$, forbidden by \Cref{lem:forbidden words}.
        \end{itemize}
        That is, $\sigma\tau$ falls into case (1) and $\tau\sigma$ falls into case (3).
        \item Let $j=k-1$ for $3\leq k\leq r-2$. Then \begin{itemize}
            \item $\sigma\tau= (s_{k})(s_{k+1}s_{k-1}s_{k})$ contains $s_{k}s_{k-1}s_{k+1}$, forbidden by \Cref{lem:forbidden words}, and
            \item $\tau\sigma= (s_{k+1}s_{k-1}s_{k})(s_{k}) = s_{k+1}s_{k-1}$ is a product of independent elements of $S$, with $2=\ell(s_{k+1})+\ell(s_{k-1})<\ell(\tau)+\ell(\sigma)=4$.
        \end{itemize}
        That is, $\sigma\tau$ falls into case (3) and $\tau\sigma$ falls into case (2).
        \item Let $j=k$ for $2\leq k\leq r-3$. Then \begin{itemize}
            \item $\sigma\tau= (s_{k})(s_{k+2}s_{k}s_{k+1}) = s_{k+2}s_{k+1}\in S$, and
            \item $\tau\sigma= (s_{k+2}s_{k}s_{k+1})(s_{k})=s_{k+2}s_{k+1}s_ks_{k+1}$ contains $s_{k+2}s_{k+1}s_k$, forbidden by \Cref{lem:forbidden words}.
        \end{itemize}
        That is, $\sigma\tau$ falls into case (1) and $\tau\sigma$ falls into case (3).
        \item Let $j=k+1$ for $1\leq k\leq r-4$. Then \begin{itemize}
            \item $\sigma\tau= (s_{k})(s_{k+3}s_{k+1}s_{k+2})=s_{k+3}s_ks_{k+1}s_{k+2}$ contains $s_ks_{k+1}s_{k+2}$, forbidden by \Cref{lem:forbidden words}, and
            \item $\tau\sigma= (s_{k+3}s_{k+1}s_{k+2})(s_{k})$ contains $s_{k+1}s_ks_{k+2}$, forbidden by \Cref{lem:forbidden words}.
        \end{itemize}
        That is, both $\sigma\tau$ and $\tau\sigma$ fall into case (3).
    \end{itemize}
\smallskip

\noindent \textbf{Products of the form (b)(c) (equivalent to (c)(b) via $f$):}
    Let $\sigma=s_{k+1}s_k$ for some $2\leq k\leq r-2$. For $\tau=s_js_{j+1}$ with $2\leq j\leq r-2$ not to be independent with $\sigma$, we must have $j\in\{k-2,k-1,k,k+1,k+2\}$.
    \begin{itemize}
        \item Let $j=k-2$ for $4\leq k\leq r-2$. Then $\sigma\tau= (s_{k+1}s_k)(s_{k-2}s_{k-1})$ is forbidden by \Cref{Highest:1234}. That is, $\sigma\tau$ falls into case (3).
        \item Let $j=k-1$ for $3\leq k\leq r-2$. Then $\sigma\tau= (s_{k+1}s_k)(s_{k-1}s_{k})$ contains $(s_{k+1}s_ks_{k-1})$, forbidden by \Cref{lem:forbidden words}. That is, $\sigma\tau$ falls into case (3).
        \item Let $j=k$ for $2\leq k\leq r-2$. Then $\sigma\tau= (s_{k+1}s_k)(s_{k}s_{k+1})=1$, a product of an empty subset of $S$, {with $0=\ell(1)<\ell(s_{k+1}s_k)+\ell(s_{k}s_{k+1})=4$.} That is, $\sigma\tau$ falls into case (2).
        \item Let $j=k+1$ for $2\leq k\leq r-3$. Then $\sigma\tau= (s_{k+1}s_k)(s_{k+1}s_{k+2}) $ contains $s_ks_{k+1}s_{k+2}$, forbidden by \Cref{lem:forbidden words}. That is, $\sigma\tau$ falls into case (3).
        \item Let $j=k+2$ for $2\leq k\leq r-4$. Then $\sigma\tau= (s_{k+1}s_k)(s_{k+2}s_{k+3})$ contains $s_{k+1}s_ks_{k+2}$, forbidden by \Cref{lem:forbidden words}. That is, $\sigma\tau$ falls into case (3).
    \end{itemize}
    
\smallskip

\noindent \textbf{Products of the form (b)(d) or (d)(b) (equivalent to (c)(d) or (d)(c) via $f$):}
    Let $\sigma=s_{k+1}s_k$ for some $2\leq k\leq r-2$. For $\tau=s_js_{j+1}s_j$ with $2\leq j\leq r-2$ not to be independent with $\sigma$, we must have $j\in\{k-2,k-1,k,k+1,k+2\}$.
    \begin{itemize}
        \item Let $j=k-2$ for $4\leq k\leq r-2$. Then \begin{itemize}
            \item $\sigma\tau=(s_{k+1}s_k)(s_{k-2}s_{k-1}s_{k-2})$ contains $s_{k+1}s_ks_{k-2}s_{k-1}$, forbidden by \Cref{Highest:1234}, and
            \item $\tau\sigma=(s_{k-2}s_{k-1}s_{k-2})(s_{k+1}s_k)$ contains $s_{k-1}s_{k-2}s_{k+1}s_k$, forbidden by \Cref{Highest:1234}.
        \end{itemize}
        That is, both $\sigma\tau$ and $\tau\sigma$ fall into case (3).
        \item Let $j=k-1$ for $3\leq k\leq r-2$. Then \begin{itemize}
            \item $\sigma\tau=(s_{k+1}s_k)(s_{k-1}s_ks_{k-1})$ contains $s_{k+1}s_ks_{k-1}$, forbidden by \Cref{lem:forbidden words} and
            \item $\tau\sigma=(s_{k-1}s_ks_{k-1})(s_{k+1}s_k)$ contains $s_ks_{k-1}s_{k+1}$, forbidden by \Cref{lem:forbidden words}.
        \end{itemize}
        That is, both $\sigma\tau$ and $\tau\sigma$ fall into case (3).
        \item Let $j=k$ for $2\leq k\leq r-2$. Then \begin{itemize}
            \item $\sigma\tau=(s_{k+1}s_k)(s_{k}s_{k+1}s_k)=s_k\in S$, and
            \item $\tau\sigma=(s_{k}s_{k+1}s_k)(s_{k+1}s_k)=s_{k+1}\in S$.
        \end{itemize}
        That is, both $\sigma\tau$ and $\tau\sigma$ fall into case (1).
        \item Let $j=k+1$ for $2\leq k\leq r-3$. Then \begin{itemize}
            \item $\sigma\tau=(s_{k+1}s_k)(s_{k+1}s_{k+2}s_{k+1}) = s_ks_{k+1}s_ks_{k+2}s_{k+1}$ contains $s_{k+1}s_ks_{k+2}$, forbidden by \newline \Cref{lem:forbidden words}, and
            \item $\tau\sigma=(s_{k+1}s_{k+2}s_{k+1})(s_{k+1}s_k) = s_{k+1}s_{k+2}s_k$ is forbidden by \Cref{lem:forbidden words}.
        \end{itemize}
        That is, both $\sigma\tau$ and $\tau\sigma$ fall into case (3).
        \item Let $j=k+2$ for $2\leq k\leq r-4$. Then \begin{itemize}
            \item $\sigma\tau=(s_{k+1}s_k)(s_{k+2}s_{k+3}s_{k+2})$ contains $s_{k+1}s_ks_{k+2}s_{k+3}$, forbidden by \Cref{Highest:1234}, and
            \item $\tau\sigma=(s_{k+2}s_{k+3}s_{k+2})(s_{k+1}s_k)$ contains $s_{k+3}s_{k+2}s_{k+1}s_k$, forbidden by \Cref{Highest:1234}.
        \end{itemize}
        That is, both $\sigma\tau$ and $\tau\sigma$ fall into case (3).
    \end{itemize}

\smallskip
\noindent \textbf{Products of the form (b)(e) or (e)(b) (equivalent to (c)(e) or (e)(c) via $f$):}
    Let $\sigma=s_{k+1}s_k$ for some $2\leq k\leq r-2$. For $\tau=s_{j+2}s_js_{j+1}$ with $2\leq j\leq r-3$ not to be independent with $\sigma$, we must have $j\in\{k-3,k-2,k-1,k,k+1,k+2\}$.
        \begin{itemize}
        \item Let $j=k-3$ for $5\leq k\leq r-2$. Then \begin{itemize}
            \item $\sigma\tau=(s_{k+1}s_k)(s_{k-1}s_{k-3}s_{k-2})$ contains $s_ks_{k-1}s_{k-3}s_{k-2}$, forbidden by \Cref{Highest:1234}, and
            \item $\tau\sigma= (s_{k-1}s_{k-3}s_{k-2})(s_{k+1}s_k)= s_{k-3}s_{k-1}s_{k-2}s_{k+1}s_k$ contains $s_{k-1}s_{k-2}s_{k+1}s_k$, forbidden by \Cref{Highest:1234}.
        \end{itemize}
        That is, both $\sigma\tau$ and $\tau\sigma$ fall into case (3).
        \item Let $j=k-2$ for $4\leq k\leq r-2$. Then \begin{itemize}
            \item $\sigma\tau=(s_{k+1}s_k)(s_{k}s_{k-2}s_{k-1})= (s_{k+1})(s_{k-2}s_{k-1})$ is a pairwise independent product with $3=\ell(s_{k+1})+\ell(s_{k-2}s_{k-1})<\ell(\sigma)+\ell(\tau)=5$, and
            \item $\tau\sigma= (s_{k}s_{k-2}s_{k-1})(s_{k+1}s_k)$ contains $s_{k}s_{k-2}s_{k-1}s_{k+1}$, forbidden by \Cref{Highest:1234}.
        \end{itemize}
        That is, $\sigma\tau$ falls into case (2) and $\tau\sigma$ falls into case (3).
        \item Let $j=k-1$ for $3\leq k\leq r-2$. Then \begin{itemize}
            \item $\sigma\tau=(s_{k+1}s_k)(s_{k+1}s_{k-1}s_{k}) = s_ks_{k+1}s_ks_{k-1}s_k$ contains $s_{k+1}s_ks_{k-1}$, forbidden by \Cref{lem:forbidden words}, and
            \item $\tau\sigma= (s_{k+1}s_{k-1}s_{k})(s_{k+1}s_k)$ contains $s_{k-1}s_{k}s_{k+1}$, forbidden by \Cref{lem:forbidden words}.
        \end{itemize}
        That is, both $\sigma\tau$ and $\tau\sigma$ fall into case (3).
        \item Let $j=k$ for $2\leq k\leq r-3$. Then \begin{itemize}
            \item $\sigma\tau=(s_{k+1}s_k)(s_{k+2}s_{k}s_{k+1})$ contains $s_{k+1}s_ks_{k+2}$, forbidden by \Cref{lem:forbidden words}.
            \item $\tau\sigma= (s_{k+2}s_{k}s_{k+1})(s_{k+1}s_k)=s_{k+2}\in S$.
        \end{itemize}
        That is, $\sigma\tau$ falls into case (3) and $\tau\sigma$ falls into case (1).
        \item Let $j=k+1$ for $2\leq k\leq r-4$. Then \begin{itemize}
            \item $\sigma\tau=(s_{k+1}s_k)(s_{k+3}s_{k+1}s_{k+2})$ contains $s_ks_{k+3}s_{k+1}s_{k+2}$, forbidden by \Cref{Highest:1234}, and
            \item $\tau\sigma= (s_{k+3}s_{k+1}s_{k+2})(s_{k+1}s_k)$ contains $s_{k+2}s_{k+1}s_k$, forbidden by \Cref{lem:forbidden words}.
        \end{itemize}
        That is, both $\sigma\tau$ and $\tau\sigma$ fall into case (3).
    \end{itemize}
\smallskip

\noindent \textbf{Products of the form (d)(e) or (e)(d):}
    Let $\sigma=s_ks_{k+1}s_k$ for some $2\leq k\leq r-2$. For $\tau=s_{j+2}s_js_{j+1}$ with $2\leq j\leq r-3$ not to be independent with $\sigma$, we must have $j\in\{k-3,k-2,k-1,k,k+1,k+2\}$.
        \begin{itemize}
        \item Let $j=k-3$ for $5\leq k\leq r-2$. Then \begin{itemize}
            \item $\sigma\tau=(s_ks_{k+1}s_k)(s_{k-1}s_{k-3}s_{k-2})$ contains $s_ks_{k-1}s_{k-3}s_{k-2}$, forbidden by \Cref{Highest:1234}, and
            \item $\tau\sigma=(s_{k-1}s_{k-3}s_{k-2})(s_ks_{k+1}s_k)$ contains $s_{k-1}s_{k-3}s_{k-2}s_k$, forbidden by \Cref{Highest:1234}.
        \end{itemize}
        That is, both $\sigma\tau$ and $\tau\sigma$ fall into case (3).
        \item Let $j=k-2$ for $4\leq k\leq r-2$. Then \begin{itemize}
            \item $\sigma\tau=(s_ks_{k+1}s_k)(s_{k}s_{k-2}s_{k-1})=s_ks_{k+1}s_{k-2}s_{k-1}$ is forbidden by \Cref{Highest:1234}, and
            \item $\tau\sigma=(s_{k}s_{k-2}s_{k-1})(s_ks_{k+1}s_k)$ contains $s_{k-2}s_{k-1}s_ks_{k+1}$, forbidden by \Cref{Highest:1234}.
        \end{itemize}
        That is, both $\sigma\tau$ and $\tau\sigma$ fall into case (3).
        \item Let $j=k-1$ for $3\leq k\leq r-2$. Then \begin{itemize}
            \item $\sigma\tau=(s_ks_{k+1}s_k)(s_{k+1}s_{k-1}s_{k}) = (s_{k+1}s_ks_{k+1})(s_{k+1}s_{k-1}s_k) = s_{k+1}s_ks_{k-1}s_k$ contains \newline $s_{k+1}s_ks_{k-1}$, forbidden by \Cref{lem:forbidden words}, and
            \item $\tau\sigma=(s_{k+1}s_{k-1}s_{k})(s_ks_{k+1}s_k) = s_{k+1}s_{k-1}s_{k+1}s_k = s_{k-1}s_k\in S$.
        \end{itemize}
        That is, $\sigma\tau$ falls into case (3) and $\tau\sigma$ falls into case (1).
        \item Let $j=k$ for $2\leq k\leq r-3$. Then \begin{itemize}
            \item $\sigma\tau=(s_ks_{k+1}s_k)(s_{k+2}s_{k}s_{k+1})$ contains $s_{k+1}s_ks_{k+2}$, forbidden by \Cref{lem:forbidden words}, and
            \item $\tau\sigma=(s_{k+2}s_{k}s_{k+1})(s_ks_{k+1}s_k)=(s_{k+2}s_ks_{k+1})(s_{k+1}s_ks_{k+1}) = s_{k+2}s_{k+1}\in S$.
        \end{itemize}
        That is, $\sigma\tau$ falls into case (3) and $\tau\sigma$ falls into case (1).
        \item Let $j=k+1$ for $2\leq k\leq r-4$. Then \begin{itemize}
            \item $\sigma\tau=(s_ks_{k+1}s_k)(s_{k+3}s_{k+1}s_{k+2})$ contains $s_ks_{k+3}s_{k+1}s_{k+2}$, forbidden by \Cref{Highest:1234}, and
            \item $\tau\sigma=(s_{k+3}s_{k+1}s_{k+2})(s_ks_{k+1}s_k)$ contains $s_{k+3}s_{k+1}s_{k+2}s_k$, forbidden by \Cref{Highest:1234}.
        \end{itemize}
        That is, both $\sigma\tau$ and $\tau\sigma$ fall into case (3).
        \item Let $j=k+2$ for $2\leq k\leq r-5$. Then \begin{itemize}
            \item $\sigma\tau=(s_{k+1}s_{k}s_{k+1})(s_{k+4}s_{k+2}s_{k+3})$ contains $s_{k+1}s_{k+4}s_{k+2}s_{k+3}$, forbidden by \Cref{Highest:1234}, and
            \item $\tau\sigma=(s_{k+4}s_{k+2}s_{k+3})(s_{k+1}s_ks_{k+1})$ contains $s_{k+4}s_{k+2}s_{k+3}s_{k+1}$, forbidden by \Cref{Highest:1234}.
        \end{itemize}
        That is, both $\sigma\tau$ and $\tau\sigma$ fall into case (3).
    \end{itemize}
   This completes the proof.
\end{proof}

\end{appendices}

\bibliographystyle{abbrv}
\bibliography{bibliography}

\end{document}